\pdfoutput=1
\documentclass{scrartcl}
\title{Lifting spectral triples to \\ noncommutative principal bundles}
\author{Kay Schwieger and Stefan Wagner}
\date{}

\KOMAoptions{parskip=half,paper=a4,abstract=on}

\usepackage[T1]{fontenc}
\usepackage{lmodern}
\usepackage{microtype}		
\usepackage[english]{babel}
\usepackage{amsmath, amsthm, amssymb, mathtools}
\usepackage{enumitem}
	\newlist{equivalence}{enumerate}{1}
	\setlist[equivalence]{label=(\alph*)} 
\usepackage{hyperref}
\usepackage{url}
\usepackage{xspace}
\usepackage{tikz}
\usepackage{tabu}
\usetikzlibrary{topaths}
\tikzstyle{every picture}+=[remember picture,inner xsep=0,inner ysep=0.25ex]
\usepackage{booktabs}

\usepackage{fixme}
\newtheoremstyle{plain}
  {\topsep}   
  {\topsep}   
  {\itshape}  
  {0pt}       
  {\bfseries\sffamily} 
  {.}         
  {5pt plus 1pt minus 1pt} 
  {}          
  
  \newtheoremstyle{definition}
  {\topsep}   
  {\topsep}   
  {\normalfont}  
  {0pt}       
  {\bfseries\sffamily} 
  {.}         
  {5pt plus 1pt minus 1pt} 
  {}          
  
  \newtheoremstyle{remark}
  {0.5\topsep}   
  {0.5\topsep}   
  {\normalfont}  
  {0pt}       
  {\itshape} 
  {.}         
  {5pt plus 1pt minus 1pt} 
  {}          

\theoremstyle{plain}
	\newtheorem{theorem}{Theorem}[section]
	\newtheorem{lemma}[theorem]{Lemma}
	\newtheorem{corollary}[theorem]{Corollary}
	
\theoremstyle{definition}
	\newtheorem{defn}[theorem]{Definition}
\theoremstyle{remark}
	\newtheorem{remark}[theorem]{Remark}
	\newtheorem{example}[theorem]{Example}

\newcommand*{\N}{\mathbb{N}}		
\newcommand*{\Z}{\mathbb{Z}}		
\newcommand*{\R}{\mathbb{R}}		
\newcommand*{\C}{\mathbb{C}}		
\newcommand*{\Unitary}{\mathcal{U}} 
\newcommand*{\tensor}{\otimes}		

\DeclarePairedDelimiter{\scal}{\langle}{\rangle}	
\DeclarePairedDelimiter{\norm}{\lVert}{\rVert} 	
\DeclareMathOperator{\Span}{span}
\DeclareMathOperator{\id}{id}		
\DeclareMathOperator{\Ad}{Ad}		
\DeclareMathOperator{\ev}{\text{ev}}		
\DeclareMathOperator{\SU}{SU}
\DeclareMathOperator{\Aut}{Aut}

\newcommand{\cf}{\mbox{cf.}\xspace}				
\newcommand{\eg}{\mbox{e.\,g.}\xspace}			
\newcommand*{\etal}{\mbox{et.\,al.}\xspace}		
\newcommand*{\ie}{\mbox{i.\,e.}\xspace}			
\newcommand{\wrt}{\mbox{w.\,r.\,t.}\xspace}		
\newcommand*{\Star}{$^*$\nobreakdash}

\newcommand*{\hilb}{\mathfrak}		
\newcommand*{\alg}{\mathcal}		
\newcommand*{\hH}{\hilb H}			
\newcommand*{\one}{1}				
\newcommand*{\aA}{\alg{A}}			
\newcommand*{\aB}{\alg{B}} 			
\newcommand*{\DD}{\mathbb D}		
\newcommand*{\End}{\mathcal L}	
\newcommand*{\Com}{\mathcal K}	
	
\newcommand*{\Mul}{\mathcal M}
\newcommand*{\Cont}{C}
\newcommand*{\CL}{\C\ell}
\newcommand*{\spin}{\text{spin}}

\DeclareMathOperator{\dom}{dom} 				
\DeclareMathOperator{\Lie}{L}					
\DeclareMathOperator{\lin}{span}				
\DeclarePairedDelimiterX{\lprod}[2]{{\,\prescript{}{#1}{\langle}}}{\rangle}{#2}
\DeclarePairedDelimiterX{\rprod}[2]{\langle}{\rangle_{#1}}{#2}
\DeclarePairedDelimiterX{\ketbra}[2]{\lvert}{\rvert}{#1 \delimsize\rangle \delimsize\langle #2}
\newcommand*{\dt}{\tfrac{d}{dt}\Big|_{t=0}} 	
\DeclareMathOperator{\Irrep}{Irr} 			

\newcommand*{\acts}{\,.\,}

\newcommand{\doubleitem}{%
  \begingroup
  \stepcounter{enumi}%
  \edef\tmp{\theenumi,}%
  \stepcounter{enumi}
  \edef\tmp{\endgroup\noexpand\item[\tmp\labelenumi]}%
  \tmp}
  
 \newcommand{\tripleitem}{%
 \begingroup \stepcounter{enumi}%
 \edef\tmp{\theenumi,}%
 \stepcounter{enumi} 
 \edef\tmpt{\theenumi,} 
 \stepcounter{enumi} 
 \edef\tmp{\endgroup\noexpand\item[\tmp\tmpt\labelenumi]}%
 \tmp}

\begin{document}

\author{
	Kay Schwieger \thanks{
		iteratec GmbH, Stuttgart, 
		\href{mailto:kay.schwieger@gmail.com}{\nolinkurl{kay.schwieger@gmail.com}}
	} \and 
	Stefan Wagner \thanks{
		Blekinge Tekniska H\"ogskola,
		\href{mailto:stefan.wagner@bth.se}{\nolinkurl{stefan.wagner@bth.se}}
	}
}
\sloppy
\maketitle

\begin{abstract}
	\noindent
	Given a free action of a compact Lie group $G$ on a unital C\Star-algebra $\aA$ and a spectral triple on the corresponding fixed point algebra $\aA^G$, we present a systematic and in-depth construction of a spectral triple on $\aA$ that is build upon the geometry of $\aA^G$ and $G$. We compare our construction with a selection of established examples.

	\vspace*{0.3cm}

\end{abstract}

\listoffixmes

\section{Introduction}\label{sec:intro}

The tremendous work of Hopf, Stiefel, and Whitney in the 1930's demonstrated the importance of principal bundles for various applications to algebraic topology, geometry and mathematical physics. 
In the noncommutative setting the notion of a free action of a quantum group on a C\Star-algebra provides a natural framework for noncommutative principal bundles (see, \eg, \cite{BaCoHa15,Ell00,Phi87,SchWa17} and ref.~therein). 
Their structure theory and their relation to $K$-theory (see, \eg, \cite{Bre04,CoYa13a,DGH01,ForRen19,SchWa15,SchWa16,SchWa17} and ref.~therein) certainly appeal to operator algebraists and functional analysts. 
In addition, noncommutative principal bundles are becoming increasingly prevalent in various applications of geometry (\cf~\cite{IoMa16,Ivan17,Meir18,SchWa18}) and mathematical physics (see, \eg, \cite{BaHa-etal07,CacMes19,DaSi13,DaSiZu14,EchNeOy09,HaMa10,LaSu05,Wahl10} and ref.~therein). 

Geometric aspects of noncommutative principal bundles, however, have not been studied yet in a cohesive way. 
Using a universal differential calculus, the algebraic framework of Hopf-Galois extensions yields abstract notions of connections and curvature (see, \eg,~\cite{DGH01,gracia2000,LaSu05} and ref.~therein). 
In this article we take a different approach and investigate the geometry of principal bundles by means of Connes' spectral triples. 

Spectral triples lay the foundation for noncommutative geometric spaces. 
Along with some additional structure they allow to extend many techniques from Riemannian spin geometry (see, \eg,~\cite{Co94,Co08,gracia2000,Var06} and ref.~therein).  
Another reason for looking at spectral triples is the role they play in mathematical physics. 
For instance, they provide a noncommutative formulation of the standard model of particle physics and of the integrality of the quantum Hall current (\cf~\cite{ChaCo08,CoMa08b,McCa97}). 
Spectral triples and their applications in quantum field theory is also an area of active research (see, \eg,~\cite{AaGr17,CoMa08b} and ref.~therein).

The purpose of this paper is to give a methodical construction of spectral triples on noncommutative principal bundles that are build upon the geometry of the ``quantum base spaces'' and the underlying structure groups. 
Since the subject of noncommutative principal bundles or, equivalently, free actions is better understood in the compact case (see,~\eg,~\cite{BaCoHa15,CoYa13a,SchWa15,SchWa17} and ref.~therein), we will restrict our study to free actions of compact Lie groups on unital C\Star-algebras. 
More precisely, for a free action of a compact Lie group $G$ on a unital C\Star-algebra $\aA$ and a spectral triple on the corresponding fixed point algebra~$\aA^G$, we provide a systematic construction of a spectral triple on $\aA$ by means of the geometry of $\aA^G$ and $G$. 
To the best of our knowledge, such spectral triples on nontrivial noncommutative principal bundles with higher dimension non-Abelian structure groups have not been worked out yet.

\subsection*{Survey of the field}

Let us briefly summarize the current research frontiers covered by this article. 
Regarding the notation, we refer the reader to Section~\ref{sec:pre_not} below.

Let $\aB$ be a C\Star-algebra, let $\DD$ be a spectral triple on $\aB$, and let $\alpha$ be a \Star-automorphism of~$\aB$. 
Bellissard, Marcolli, and Reihani established in their seminal work~\cite{BeMaRe10} that there is a canonical spectral triple on the crossed product $\aB \rtimes_\alpha \mathbb{Z}$ that is built upon $\DD$ and a differential operator on the circle (via Fourier transform) and characterizes the metric properties of the C\Star-dynamical system $(\aB, \mathbb{Z}, \alpha)$. 
The Bellissard-Marcolli-Reihani theory was developed further in the articles~\cite{HaSkWhYa13,IoMa16,Pat14} for actions of more general classes of groups (\eg, discrete and second-countable locally compact). 
In this context we would like to recall that, given a C\Star-dynamical system $(\aA,G,\alpha)$ with a countable discrete Abelian group~$G$, the natural dual action of the compact dual group $\hat G$ on the crossed product $\aA \rtimes_\alpha G$ is free (\cf~\cite[Sec.~4]{SchWa16}). 
For this reason crossed products provide a natural source of noncommutative principal bundles. 

Ammann and B\"ar~\cite{Amm98,AmmBae98} looked into the properties of the Riemannian spin geometry of a smooth principal $U(1)$-bundle. 
Under suitable hypotheses they related the Riemannian spin geometry of the total space to the Riemannian spin geometry of the base space. 
A~noncommutative generalization of these results for the Hopf-Galois analogue of principal torus bundles was developed about a decade later by Dabrowski, Sitarz, and Zucca~\cite{DaSi13,DaSiZu14}. 
Of particular interest is also Zucca's unpublished PhD-thesis~\cite[Sec.~8]{Zucca}, in which he presented conditions in order to build a real spectral triple for cleft extensions of compact connected semisimple Lie groups. 
Moreover, Sitarz and Venselaar~\cite{SiVe15} studied spectral triples on quantum lens spaces as orbit spaces of free actions of cyclic groups on the spectral geometry of $\SU_q(2)$. 
Aiello, Guido, and Isola~\cite{AGI17} provided examples of noncommutative coverings and extended spectral triples on the base space to spectral triples on the inductive family of coverings in such a way that the covering projections are locally isometric.

We would also like to mention some connections to unbounded $KK$-theory, which was developed over the last decade by Kaad, Lesch, Mesland \etal 
Indeed, Gabriel and Grensing~\cite{GaGr13} proposed a construction of spectral triples for a class of crossed product-like algebras that gives an unbounded representative of the Kasparov product of the original spectral triple and the Pimsner-Toeplitz extension associated with the crossed product by a Hilbert module. 
Forsyth and Rennie~\cite{ForRen19} provided sufficient conditions to factorize an equivariant spectral triple as a Kasparov product of unbounded classes constructed from the compact Abelian group action on the algebra and from the fixed point spectral triple. 
In addition, they showed that each equivariant Dirac-type spectral triple on the total space of a torus principal bundle factorizes and that the Kasparov product~\cite{KaLesch13,Mes16} can be used to recover the original triple. 
Kaad and van Suijlekom~\cite{KaSui18a,KaSui20} established that the Dirac operator on the total space of an almost-regular fibration can be written, up to an explicit ``obstructing'' curvature term, as the tensor sum of a vertically elliptic family of Dirac operators with the horizontal Dirac operator representing the interior Kasparov product in bivariant $K$-theory (see~\cite{KaSui18b} for a similar discussion in the context of toric noncommutative manifolds). 

Finally, we would like to draw attention to a recent paper by \'{C}a\'{c}i\'{c} and Mesland~\cite{CacMes19}. There the authors presented a new, general approach to gauge theory on principal $G$-spectral triples, where $G$ is a compact connected Lie group. 
It is our hope that this work will contribute to the development and understanding of a noncommutative gauge theory, for instance, by providing a geometric oriented notion of ``parallel transport'' or, more generally, of ``parallelity'' on a noncommutative principal bundle - a concept whose importance cannot be overemphasized.


\subsection*{Organization of the article}

After some preliminaries, we present in Section~\ref{sec:freeness}, Lemma~\ref{lem:coisometry}, yet another characterization of freeness for C\Star-dynamical systems, which is well-adapted to our purposes. 
Given a free C\Star-dynamical system $(\aA,G,\alpha)$, we make use of this characterization in Section~\ref{sec:representation} to get a faithful covariant representation of $(\aA,G,\alpha)$ on some suitable amplification of its fixed point algebra $\aA^G$ (Lemma~\ref{lem:pi_S}). 
We also classify all faithful covariant representations of $(\aA,G,\alpha)$ up to unitary equivalence (Theorem~\ref{thm:class_cov_rep}). 
As a further application we show that the isotypic components of~$\aA$ act nondegenerately (Corollary~\ref{cor:Sunitary}).

Now, let $(\aA,G,\alpha)$ be a free C\Star-dynamical system with a compact Lie group $G$ and let $\DD$ be a spectral triple on $\aA^G$. 
Section~\ref{sec:permanence}, the main body of this article, is devoted to the construction of a spectral triple on $\aA$ in terms of $\DD$ and the geometry of $G$.
As a preliminary step we briefly recall the notion of a factor system. 
Our procedure then naturally splits into the following five main steps:
\begin{enumerate}
	\item
		We form a ``quantum manifold'' or, in greater detail, a dense unital \Star-subalgebra $\aA_0 \subseteq\aA$ of smooth elements (Theorem~\ref{thm:A_0}).
	\item
		We extend the faithful \Star-representation of $\aA^G$ to a faithful covariant representation of $(\aA,G,\alpha)$ (Theorem~\ref{thm:lift_cov_rep}). 
	\item
		We lift the Dirac operator on $\aA^G$ to a Dirac-type operator $D_h$ on $\aA$ (Corollary~\ref{cor:D_h}) and show that $D_h$ has, under some suitable conditions, bounded commutators with $\aA_0$ (Theorem~\ref{thm:D_hcomm}). The operator $D_h$ later becomes the horizontal component of a Dirac operator.
	\item
		Using the group action, we construct a Dirac operator $D_v$, which later becomes the vertical component of the Dirac operator (Corollary~\ref{cor:vertop}).
	\item
		Finally, it remains to tie everything together. More precisely, we suitably assemble the operators $D_h$ and $D_v$ to a Dirac operator $D_\aA$ on $\aA$ (Theorem~\ref{thm:main}), and in this way we get a spectral triple on $\aA$ (Corollary~\ref{cor:main}).
\end{enumerate}
 
The remaining Sections,~\ref{sec:crossprod},~\ref{sec:QT} and~\ref{sec:Homogeneous}, are devoted to thoroughly treating examples or, more precisely, to investigating how our construction compares to well-established examples. 
Indeed, in Section~\ref{sec:crossprod} we show that our construction generalizes the construction given in the pioneering article~\cite[Sec.~3.4]{BeMaRe10} by Bellissard, Marcolli, and Reihani.
In Section~\ref{sec:QT} and Section~\ref{sec:Homogeneous} we explore our construction for a free 2-torus action on the quantum 4-torus and for homogeneous spaces, respectively. 
The analysis in Section~\ref{sec:QT} is straightforward and we recover the standard Dirac operator on the quantum 4-torus. The computations in Section~\ref{sec:Homogeneous} are more technical and we find that our construction differs from the standard Dirac operator on the group by a central term (Theorem~\ref{thm:compare} and Remark~\ref{rem:compare}).

Last but not least, we would like to mention that we have put in enough detail so that this paper will be accessible to a broad readership.

\section{Preliminaries and notation}\label{sec:pre_not}

Our study revolves around permanence properties of noncommutative principal bundles with respect to spectral triples. Along the way we use various tools from geometry and operator algebras. This preliminary section exhibits the most fundamental definitions and notations in use.

To begin with, we provide some standard references. For a recent account of the theory of spectral triples and, more generally, of noncommutative geometry we refer to the excellent expositions~\cite{gracia2000,Var06} by Gracia-Bond{\'\i}a, Figueroa, and Varilly. Our standard references for operator algebras are the opuses~\cite{BB06,Ped18} by Blackadar and Pedersen, respectively. For a thorough treatment of Hilbert module structures we refer to the book~\cite{Rae98} by Raeburn and Williams and the memoirs~\cite{EKQR06} by Echterhoff \etal

\paragraph*{Hilbert spaces}
All Hilbert spaces are assumed to be complex and come equipped with an inner product that is linear in the second component if not explicitly mentioned otherwise.

\paragraph*{About tensor products}
In this paper tensor products of C\Star-algebras are taken with respect to the minimal tensor product, which is simply denoted by $\tensor$. 
Let $\aA$, $\aB$, and $\alg C$ be unital C\Star-algebras. If there is no ambiguity, then we consider each one of them as a C\Star-subalgebra of $\aA \tensor \aB \tensor \alg C$ and extend maps on $\aA$, $\aB$, or $\alg C$ canonically by tensoring with the respective identity map. 
For the sake of clarity, we may make use of the leg numbering notation, for instance, given $x \in \aA \tensor \alg C$, we write $x_{13}$ to denote the corresponding element in~$\aA \tensor \aB \tensor \alg C$.

\paragraph*{About groups}\label{sec:groups}
Let $G$ be a compact group. 
We write $\Irrep(G)$ for the set of equivalence classes of irreducible representations of $G$ and let $\one \in \Irrep(G)$ stand for the class of the trivial representation.  
There are two C\Star-algebras associated with $G$: the algebra $\Cont(G)$ of continuous complex-valued functions on $G$ and the (reduced) group C\Star-algebra $C^*_r(G)$.
If $G$ is a Lie group, then we denote its Lie algebra by $\Lie(G)$ and the algebra of smooth functions by $\Cont^\infty(G)$.
For each $X \in \Lie(G)$ we write $\partial_X^G$ for the skew-adjoint operator on $\Cont^\infty(G)$ given by
\begin{align*}
	\partial_X^G f(g) := \dt f \bigl( \exp(-tX) g).
\end{align*}

\paragraph*{About Clifford algebras}\label{sec:cliff}
For a real Hilbert space $\hH$ we write $\CL(\hH)$ for the associated complex Clifford algebra and use the symbol ``$\cdot$'' for its product. 
We consider $\hH$ as a linear subspace of $\CL(\hH)$ and follow the convention that the defining relation is
\begin{align*}
	\xi \cdot \eta + \eta \cdot \xi = -2 \rprod{\hH}{\xi, \eta} \, \one
\end{align*}
for all $\xi, \eta \in \hH$. 
In particular, if $G$ is a compact Lie group, then we simply write $\CL(G)$ instead of $\CL\bigl( (\Lie(G) \bigr)$ for ease of notation.

\paragraph*{About unbounded operators}\label{sec:unbounded}
By an unbounded operator on a Hilbert space $\hH$ we mean a linear map $D: \dom(D) \to \hH$ defined on a dense subspace $\dom(D) \subseteq \hH$. The following results about unbounded self-adjoint operators are well-known to experts, but we could not spot a prominent reference. The proofs are postponed to Appendix~\ref{sec:proofs}.

\pagebreak[3]
\begin{lemma}\label{lem:resD}
	Let $D$ be an unbounded self-adjoint operator on a Hilbert space $\hH$ with domain $\dom(D)$ and let $p \in \End(\hH)$ be an orthogonal projection such that $p \dom(D) \subseteq \dom(D)$ and such that the commutator $[D, p]$ is bounded. Consider the unbounded operator 
	\begin{equation*}
		D_p := p D p 
		\qquad
		\text{with domain}~\dom(D_p) := p \dom(D)
	\end{equation*}
	on the Hilbert space $p(\hH)$. Then the following assertions hold:
	\begin{enumerate}[label={\arabic*}.,ref=\ref{lem:resD}.{\textit{\arabic*}}.]
	\item
		$D_p$ is self-adjoint.\label{resDsa}
	\item 
		If $D$ has compact resolvent, then $D_p$ has compact resolvent.\label{resDcomres}
	\end{enumerate}
\end{lemma}

We also make use of the following construction of an essentially self-adjoint operator from a continuous unitary representation $u:G \to \alg \Unitary(\hH)$ of a finite-dimensional compact Lie group $G$. For this, let $\hH^\infty \subseteq \hH$ stand for the subspace of smooth vectors, which is $G$\nobreakdash-invariant and dense in $\hH$ due to~\cite{Gard47} (see also~\cite{Neeb10,Neeb17}). Equip $\Lie(G)$ with an $\Ad$-invariant inner product and let $\pi_\spin:\CL(G) \to \End(\hH_\spin)$ be a finite-dimensional \Star-representation of the Clifford algebra $\CL(G)$. Finally, for each $X \in \Lie(G)$ set $F_X := \pi_\spin(X)$ and denote by  $\partial_X u$ the unbounded operator on $\hH$ with domain $\dom(\partial_X u) := \hH^\infty$ defined as
\begin{align*}
	\partial_X u(\eta) := \lim_{t \to 0} \frac{u_{\exp(tX)}(\eta)-\eta}{t}.
\end{align*}
Then~\cite[Prop.~4.1]{GaGr16} establishes that for any orthonormal basis $X_1, \dots, X_n$ of $\Lie(G)$ we obtain an essentially self-adjoint operator with $\dom(D) := \hH^\infty \tensor \hH_\spin$ by putting
\begin{align*}
	D := \sum_{k=1}^n \partial_{X_k} u \tensor F_{X_k}.
\end{align*}

\begin{lemma}\label{lem:comres2}
	Suppose the unitary representation $u:G \to \alg \Unitary(\hH)$ has finite-dimensional multiplicity spaces. Then the closure of $D$ has compact resolvent.
\end{lemma}

\paragraph*{About spectral triples}
{\fussy Let $\aA$ be a unital C\Star-algebra. By a \emph{spectral triple} on $\aA$ we mean a quadruple $\DD_\aA = (\aA_0,\pi,\hH,D)$ consisting of a dense unital \Star-subalgebra $\aA_0 \subseteq \aA$, a~faithful \Star-representation $\pi: \aA \to \mathcal L(\hH)$ on a Hilbert space $\hH$, and a possibly unbounded self-adjoint operator $D$ on $\hH$ such that}
\begin{enumerate}
	\item
		$D$ has compact resolvent,
	\item
		$\pi(\aA_0) \dom(D) \subseteq \dom(D)$,
	\item
		and all commutators $[D,\pi(x)]$ for $x \in \aA_0$ are bounded. 
\end{enumerate}

\paragraph*{About \texorpdfstring{C$^*$}{C*}-dynamical systems}\label{sec:Cstar}
Let $\aA$ be a unital C\Star-algebra and $G$ a compact group that acts on $\aA$ by \Star-automorphisms $\alpha_g:\aA \to \aA$, $g \in G$, such that $G \times \aA \to \aA$, $(g,x) \mapsto \alpha_g(x)$ is continuous. Throughout this article we call such data a \emph{C\Star-dynamical system}, denote it briefly by $(\aA,G,\alpha)$, and typically write $\aB := \aA^G$ for its fixed point algebra. In case $G$ is a compact Lie group, we make use of the symbol $\aA^\infty \subseteq \aA$ to denote the \Star-subalgebra of smooth elements.
	
For $\sigma \in \Irrep(G)$ we let $A(\sigma)$ stand for the corresponding isotypic component of $\aA$ and regard it as a correspondence over $\aB$ in terms of the usual left and right multiplication and the right $\aB$-valued inner product given by $\rprod{\aB}{x,y} := \int_G \alpha_g(x^*y) \, dg$ for all $x,y \in A(\sigma)$. On account of the Peter-Weyl Theorem (see, \eg, \cite[Thm.~4.22]{HoMo06}) $\aA$ decomposes into its isotypic components, which amounts to saying that the algebraic direct sum $\bigoplus_{\sigma\in\Irrep(G)}^{\text{alg}} A(\sigma)$ is a dense \Star-subalgebra of~$\aA$.
	
Aside from the isotypic components, it is expedient to look at the corresponding multiplicity spaces, by which we mean the sets
\begin{align*}
	\Gamma_\aA(\sigma) := \{ x \in \aA \tensor V_\sigma : (\alpha_g \tensor \sigma_g)(x) = x \quad \forall g\in G\}, \qquad \sigma \in \Irrep(G).
\end{align*}

They also play the role of modules of sections of a vector bundle (with standard fibre $V_\sigma$) associated with $(\aA,G,\alpha)$.
Each $\Gamma_\aA(\sigma)$, $\sigma \in \Irrep(G)$, is naturally a correspondence over $\aB$ with respect to the canonical left and right action and the restriction of the right $\aA$-valued inner product on $\aA \tensor V_\sigma$ determined by $\rprod{\aA}{a \tensor v, b \tensor w} := \scal{v,w} a^*b$ for all $a,b \in \aA$ and $v,w \in V_\sigma$. 
The corresponding mapping $\sigma \mapsto \Gamma_\aA(\sigma)$ can be extended to an additive functor from the representation category of $G$ into the category of C\Star-correspondences over~$\aB$. 
Most notably, for each $\sigma \in \Irrep(G)$ we get a $G$-equivariant isomorphism
\begin{equation}\label{eq:iso_section}
	\Phi_\sigma : \Gamma_\aA(\sigma) \tensor \bar{V}_\sigma \to A(\bar{\sigma})
\end{equation}
of correspondences over $\aB$, where the $G$-action on $\Gamma_\aA(\sigma) \tensor \bar{V}_\sigma$ is the obvious one on the second tensor factor, by restricting the map 
\begin{align*}
	\aA \tensor V_\sigma \tensor \bar{V}_\sigma \to \aA, \qquad a \tensor v \tensor \bar{w} \mapsto \scal{w,v} \, a.
\end{align*}
More generally, if $(\sigma,V_\sigma)$ is a finite-dimensional representation of $G$ and $\sigma = \sigma_1 \oplus \dots \oplus \sigma_n$ is a decomposition into irreducible subrepresentations, then the same recipe as above gives a $G$-equivariant morphism $\Phi_\sigma : \Gamma_\aA(\sigma) \tensor \bar{V}_\sigma \to A(\bar{\sigma})$ of correspondences over $\aB$, where $A(\bar{\sigma}) := \bigoplus_{i = 1}^n A(\bar{\sigma_i}) \subseteq \aA$. 
The functor $\sigma \mapsto \Gamma_\aA(\sigma)$ together with the family of transformations
\begin{equation}
	\Gamma_\aA(\sigma) \tensor_\aB \Gamma_\aA(\tau) \to \Gamma_\aA(\sigma \tensor \tau), \qquad x \tensor y \mapsto x_{12} y_{13}\label{eq:gammamulti}
\end{equation}
for all finite-dimensional representations $\sigma$ and $\tau$ of $G$ constitute a so-called weak tensor functor and allows to reconstruct the reduced form of the C\Star-dynamical system $(\aA,G,\alpha)$ up to isomorphism (see~\cite[Sec.~2]{Ne13}).

\section{Free \texorpdfstring{C$^*$}{C*}-dynamical systems}\label{sec:freeness}

A C\Star-dynamical system $(\aA,G,\alpha)$ is called \emph{free} if the so-called \emph{Ellwood map} 
\begin{align*}
	\Phi: \aA \tensor_{\text{alg}} \aA \rightarrow C(G,\aA), \qquad \Phi(x\tensor y)(g):=x \alpha_g(y)
\end{align*}
has dense range with respect to the canonical C\Star-norm on $\Cont(G,\aA)$. This condition was originally introduced for actions of quantum groups on C\Star-algebras by Ellwood~\cite{Ell00} and is known to be equivalent to Rieffel's saturatedness~\cite{Rieffel91} and the Peter-Weyl-Galois condition~\cite{BaCoHa15}. Additionally, we recall that Phillips~\cite{Phi09} considered some stronger variants of freeness.

One of the key tools used in this article is a characterization of freeness that we provided in \cite[Lem.~3.2]{SchWa17}, namely that a C\Star-dynamical system $(\aA,G,\alpha)$ is free if and only if for each irreducible representation $(\sigma,V_\sigma)$ of $G$ there is a finite-dimensional Hilbert space~$\hH_\sigma$ and an isometry $s(\sigma) \in \aA \tensor \End(V_\sigma,\hH_\sigma)$ satisfying $\alpha_g\bigl(s(\sigma)\bigr)=s(\sigma) (\one_\aA \tensor \sigma_g)$ for all $g \in G$. However, to simplify notation we patch this family of isometries together and use the following characterization instead. 

\begin{lemma}\label{lem:coisometry}
	For a C\Star-dynamical system $(\aA, G, \alpha)$ the following statements are equivalent:
	\begin{equivalence}
		\item
			$(\aA, G, \alpha)$ is free.
		\item
			There is a unitary representation $\mu: G \to \mathcal U(\hH)$ with finite-dimensional multiplicity spaces and, given any faithful covariant representation $(\pi,u)$ of $(\aA,G,\alpha)$ on some Hilbert space $\hH_\aA$, an isometry $s \in \End(\hH_\aA \tensor L^2(G), \hH_\aA \tensor \hH)$ satisfying
			\begin{alignat}{2}
				s \aA \tensor \Com(L^2(G)) &\subseteq \aA \tensor \Com(L^2(G),\hH),\label{eq:SOPleft}
				\\
				(u_g \tensor \one_\hH) s &= s (u_g  \tensor r_g) \qquad &&\forall g \in G,\label{eq:SOPequivariance}
				\\
				(\one_\aA \tensor \mu_g) s &= s (\one_\aA \tensor \lambda_g) \qquad &&\forall g \in G.\label{eq:SOPcommuting}
			\end{alignat}
Here, we do not distinguish between $\aA$ and $\pi(\aA) \subseteq \End(\hH_\aA)$ for the sake of brevity, and the tensor product $\aA \tensor \Com(L^2(G),\hH)$ is closed with respect to the operator norm, where $\Com(L^2(G),\hH)$ is regarded as the respective corner of $\Com(L^2(G) \oplus \hH)$.
	\end{equivalence}
\end{lemma}

As we will mainly be concerned with the implication ``$(a) \Rightarrow (b)$'', we have decided to give the proof of this implication at this point only and to move the proof of the implication ``$(b) \Rightarrow (a)$'' to Appendix~\ref{sec:proofs}.

\begin{proof}
	If the C\Star-dynamical system $(\aA,G,\alpha)$ is free, then for each $\sigma \in \Irrep(G)$ there is a finite-dimensional Hilbert space $\hH_\sigma$ and an isometry $s(\sigma) \in \aA \tensor \End(V_\sigma,\hH_\sigma)$ satisfying $\alpha_g\bigl(s(\sigma)\bigr)=s(\sigma) (\one_\aA \tensor \sigma_g)$ for all $g \in G$ (\cf \cite[Lem.~3.2]{SchWa17}). To establish the claims in~(b), we consider the unitary representation $\mu: G \to \mathcal U(\hH)$ defined by
	\begin{align}\label{eq:decompH}
		\hH &:= \bigoplus_{\sigma \in \Irrep(G)} \hH_\sigma \tensor \bar V_\sigma
		&
		&\text{and}
		&
		\mu_g &: = \bigoplus_{\sigma \in \Irrep(G)} \one_{\hH_\sigma} \tensor \bar \sigma_g.
	\end{align}
	Additionally, we choose a faithful covariant representation $(\pi,u)$ of $(\aA,G,\alpha)$ on some Hilbert space $\hH_\aA$ and decompose $L^2(G) = \bigoplus_{\sigma \in \Irrep(G)} V_\sigma \tensor \bar V_\sigma$ into its isotypic components such that the left and right regular representation read as $\lambda_g = \bigoplus_{\sigma \in \Irrep(G)} \one_{V_\sigma} \tensor \bar\sigma_g$ and $r_g = \bigoplus_{\sigma \in \Irrep(G)} \sigma_g \tensor \one_{\bar V_\sigma}$ for all $g \in G$, respectively. We also patch together the isometries:
	\begin{equation} \label{eq:S}
		s := \bigoplus_{\sigma \in \Irrep(G)} s(\sigma) \tensor \one_{\bar V_\sigma} \in \End(\hH_\aA \tensor L^2(G), \hH_\aA \tensor \hH).
	\end{equation}
	By construction, $\mu: G \to \mathcal U(\hH)$ has finite-dimensional multiplicity spaces. Furthermore, straightforward computations reveal that
	\begin{align*}
		s^*s
		&= \bigoplus_{\sigma \in \Irrep(G)} s(\sigma)^* s(\sigma) \tensor \one_{\bar V_\sigma} 
		=  \one_\aA \tensor \one_G,
		\\
		(u_g \tensor \one_\hH) s (u_g^* \tensor r_g^*)
		&= \bigoplus_{\sigma \in \Irrep(G)} \alpha_g\bigl(s(\sigma)\bigr) (\one_\aA \tensor \sigma_g^*) \tensor \one_{\bar V_\sigma} = s \qquad \forall g \in G,
		\\
		\mu_g s
		&= \bigoplus_{\sigma \in \Irrep(G)} s(\sigma) \tensor \bar \sigma_g = s \lambda_g \qquad \forall g \in G.
	\end{align*}
	In other words, $s$ is an isometry satisfying the Equations~\eqref{eq:SOPequivariance} and~\eqref{eq:SOPcommuting}. It therefore remains to deal with Equation~\eqref{eq:SOPleft}. Indeed, given $a\in \aA$ and an operator $T$ on $L^2(G)$ such that $T \bigl( L^2(G) \bigr) \subseteq V_\sigma \tensor \bar V_\sigma$ for some $\sigma \in \Irrep(G)$, it is easily seen that $a \tensor T \in \aA \tensor \Com \bigl( L^2(G) \bigr)$ has the following property:
	\begin{equation*}
		s (a \tensor T) = (s(\sigma) \tensor \one_{\bar V_\sigma}) (a \tensor T) \in \aA \tensor \Com(L^2(G),\hH_\sigma \tensor \bar V_\sigma) \subseteq \aA \tensor \Com(L^2(G),\hH). 
	\end{equation*}
	But from the above we also obtain Equation~\eqref{eq:SOPleft}, because the set of such operators is total in $\aA \tensor \Com \bigl( L^2(G) \bigr)$. This completes the proof of the implication ``$(a) \Rightarrow (b)$''. 
\end{proof}

For the trivial representation of $G$ we may without loss of generality choose \mbox{$\hH_1 = \mathbb{C}$} and $s(1) = 1_\aA$. 
If this holds, then we refer to the isometry $s$ in Equation~\eqref{eq:S} as \emph{normalized} and notice that the projection $ss^*$ acts trivially on $\hH_\aA \tensor \hH_1 = \hH_\aA \tensor \C$.

\begin{remark}\label{rem:SOPright}
	In much the same way as in the proof of Lemma~\ref{lem:coisometry} we see that the adjoint $s^* \in \End \bigl( \hH_\aA \tensor \hH, \hH_\aA \tensor L^2(G) \bigr)$ of $s$ satisfies $s^* \aA \tensor \Com(\hH) \subseteq \aA \tensor \Com \bigl( \hH,L^2(G) \bigr)$ or, equivalently, that the isometry $s$ additionally satisfies
	\begin{equation*}
		\aA \tensor \Com(\hH) s \subseteq \aA \tensor \Com(L^2(G),\hH).
	\end{equation*}
	For this reason, we can assert that $s$ is a multiplier for $\aA \tensor \Com(L^2(G) \oplus \hH)$, that is, $s \in \Mul \bigl( \aA \tensor \Com( L^2(G) \oplus \hH) \bigr)$, with $(1_\aA \tensor p_\hH) s = 0 = s (1_\aA \tensor p_{L^2(G)})$, where $p_\hH$ and $p_{L^2(G)}$ denote the canonical projections onto $\hH$ and $L^2(G)$, respectively.
\end{remark}

\begin{remark}\label{rem:partial_iso}
	Let $s'(\sigma) \in \aA \tensor \End(V_\sigma,\hH'_\sigma)$, $\sigma \in \Irrep(G)$, be another family of isometries such that $\alpha_g(s'(\sigma))=s'(\sigma) (\one_\aA \tensor \sigma_g)$ for all $g \in G$. Moreover, let $\mu': G \to \mathcal U(\hH')$ and $s'$ be the corresponding unitary representation and isometry, respectively. Then there is a partial isometry $t: \hH_\aA \tensor \hH \to \hH_\aA \tensor \hH'$ satisfying $(\one_\aA \tensor \mu'_g) t = t (\one_\aA \tensor \mu_g)$ for all $g \in G$ as well as $s'(s')^* = tt^*$ and $ss^* = t^*t$ (\cf \cite[Lem.~4.3]{SchWa17}). In particular, the projections $s'(s')^*$ and $ss^*$ are Murray-von Neumann equivalent. This will be relevant later on for our attempt to characterize the faithful covariant representations of free C\Star-dynamical systems.
\end{remark}

\begin{remark}\label{rem:cleft}
	A rich class of free actions is given by so-called cleft actions (see~\cite{SchWa16}), which we now briefly recall. We say that a C\Star-dynamical system $(\aA,G,\alpha)$ is \emph{cleft} if there is a unitary $u \in \Mul \bigl( \aA \tensor C^*_r(G) \bigr)$ satisfying
	\begin{equation}
		\bar{\alpha}_g(u) = u (1_\aA \tensor r_g) \qquad \forall g \in G,\label{eq:Uequivariance}
	\end{equation}
where $\bar{\alpha}_g$ denotes the strictly continuous extension of $\alpha_g \tensor \id_{C^*_r(G)}$ to $\Mul \bigl( \aA \tensor C^*_r(G) \bigr)$. It is clear that each cleft C\Star-dynamical system is free with a possible choice for $\mu$ and $\hH$ given by $\lambda$ and $L^2(G)$, respectively. Regarded as noncommutative principal bundles, these actions are essentially characterized by the fact that all associated noncommutative vector bundles are trivial. 
\end{remark}

\section{Representations of free \texorpdfstring{C$^*$}{C*}-dynamical systems}\label{sec:representation}

In this section we look more closely into covariant representations of free C\Star-dynamical systems. 
For this purpose we fix a free C\Star-dynamical system $(\aA, G, \alpha)$ and let $\aB$ be its fixed point algebra. 

Our first goal is to characterize all faithful covariant representations of $(\aA, G, \alpha)$ up to unitary equivalence. 
For a start we consider a faithful covariant representation $(\pi,u)$ of $(\aA, G, \alpha)$ with representation space $\hH_\aA$. 
In accordance with Lemma~\ref{lem:coisometry} we choose a unitary representation $\mu: G \to \mathcal U(\hH)$ with finite-dimensional multiplicity spaces as well as an isometry $s \in \End(\hH_\aA \tensor L^2(G), \hH_\aA \tensor \hH)$ satisfying the Equations~\eqref{eq:SOPleft},~\eqref{eq:SOPequivariance}, and~\eqref{eq:SOPcommuting}. 
We assume that $s$ is normalized. 
Then $ss^*$ acts trivially on $\hH_\aA \tensor \hH_1 = \hH_\aA \tensor \C$ and, by Equation~\eqref{eq:SOPcommuting}, we have $(\one_\aA \tensor \mu_g) ss^* =  ss^* (\one_\aA \tensor \mu_g)$ for all $g \in G$. Furthermore, Equation~\eqref{eq:SOPequivariance} implies that $ss^*$ may be restricted to a projection on $\hH_\aA(\one) \tensor \hH$, where $\hH_\aA(\one)$ denotes the trivial isotypic component of $\hH_\aA$. We shall use the letter $p$ for this projection in~$\End \bigl( \hH_\aA(1) \tensor \hH \bigr)$.
	
Next, let us take into account the isometry
\begin{equation}\label{eq:J_u}
	J_u: \hH_\aA \to \hH_\aA \tensor L^2(G) = L^2(G, \hH_\aA),
	\qquad
	(J_u \eta)(g) := u_g^* \eta.
\end{equation}
We write $\hilb K_\aA$ for its range, which is the fixed point space of $\hH_\aA \tensor L^2(G)$ under the action $u_g \tensor r_g$, $g\in G$. 
Moreover, we denote by $j_\alpha$ the \Star-homomorphism
\begin{gather*}
	j_\alpha : \aA \rightarrow \Cont\bigl(G,\pi(\aA)\bigr) \subseteq \End \bigl( \hH_\aA \tensor L^2(G) \bigr), 
	\\
	j_\alpha(x)(g) = \pi\bigl( \alpha_{g^{-1}}(x) \bigr) = u_g^* \pi(x) u_g. \notag
\end{gather*}
It is evident that $j_\alpha$ is injective and that $\Ad[\one_\aA \tensor \lambda_g] \circ j_\alpha = j_\alpha \circ \alpha_g$ for all $g \in G$.

\begin{lemma}\label{lem:pi_S}
	Consider the Hilbert space $\hH_s := ss^*(\hH_\aA \tensor \hH)$ along with the map
	\begin{align}\label{eq:pi_S}
		\pi_s : \aA \rightarrow \End(\hH_s), \qquad \pi_s(x) := s j_\alpha(x) s^*.
	\end{align}
	Then $\pi_s$ is a faithful \Star-representation. Additionally, $\pi_s$ is equivariant \wrt the action $\alpha_g$, $g \in G$, on $\aA$ and the action $\Ad[\one_\aA \tensor \mu_g]$, $g\in G$, on $\End(\hH_s)$, respectively.
\end{lemma}
\begin{proof}
	Using the isometry property of $s$ and the injectivity of $j_\alpha$, we easily infer that $\pi_s$ is a well-defined faithful \Star-representation. Next, let $g \in G$ and let $x\in \aA$. Then
	\begin{equation*}
		\Ad[\one_\aA \tensor \mu_g] \bigl(\pi_s(x) \bigr) \overset{\eqref{eq:SOPcommuting}}= s \bigl( \Ad[\one_\aA \tensor \lambda_g]( j_\alpha(x) \bigr) s^* = \pi_s \bigl( \alpha_g(x) \bigr), 
	\end{equation*}
which proves that $\pi_s$ is also $G$-equivariant.
\end{proof}

Because $ss^*$ lies in  $\Mul \bigl( \aB \tensor \Com(\hH) \bigr)$ (see~Remark~\ref{rem:SOPright}) and the compact operators form an ideal, we instantly get the following statement.

\begin{corollary}\label{cor:pi_s}
	$\pi_s(\aA)$ is a subset of the multiplicator algebra $\Mul \bigl( ss^* \bigl( \aB \tensor \Com(\hH) \bigr) ss^* \bigr)$.
\end{corollary}

The preceding discussion entails that $(\aA, G, \alpha)$ can be covariantly represented as multiplicators on $\aB \tensor \Com(\hH)$. 
In consequence, any  \Star-representation $\pi_\aB : \aB \to \End(\hH_\aB)$ gives rise to a covariant representation, $\bigl( (\pi_\aB \tensor \id) \circ \pi_s,\one_\aB \tensor \mu \bigr)$, of $(\aA, G, \alpha)$ on the Hilbert space $\pi_\aB \tensor \id(ss^*)(\hH_\aB \tensor \hH)$ that has trivial isotypic component $\hH_\aB$. 
If $\pi_\aB$ is faithful, then so is the covariant representation. 
We shall now demonstrate that, in fact, each faithful covariant representation of $(\aA, G, \alpha)$ is of this form.

To begin with, we notice that each $\pi_s(x)$, $x \in \aA$, intertwines the unitary representation $G \ni g \mapsto u_g \tensor \one_\hH \in \Unitary(\hH_s)$. Hence the subspace $\hH_p := p \bigl(\hH_\aA(1) \tensor \hH \bigr) \subseteq \hH_s$ is invariant under~$\pi_s$ and, thus, we obtain a \Star-representation of $\aA$ on $\hH_p$ by putting 
\begin{equation}
	\pi_p : \aA \rightarrow \End(\hH_p), \qquad \pi_p(x) := \pi_s(x) |_{\hH_p}^{\hH_p}.\label{eq:pi_p}
\end{equation}
Our next result establishes that $\hH_p$ is isomorphic to $\hilb K_\aA$, the range of the map $J_u$ from Equation~\eqref{eq:J_u}.

\begin{lemma}\label{lem:Sunitary}
	For the map $\Phi_s: \mathfrak{K}_\aA \rightarrow \hH_p$, $\Phi_s(f) := s(f)$ the following assertions hold:
	\begin{enumerate}
		\item
			$\Phi_s$ is a unitary map such that $(\one_\aA \tensor \mu_g) \Phi_s = \Phi_s (\one_\aA \tensor \lambda_g)$ for all $g \in G$.
		\item
			$\Phi_s(\eta \tensor \one_{L^2(G)}) := \eta \tensor \one_\C$ for all $\eta \in \hH_\aA(\one)$.
		\item
			$\Phi_s j_\alpha(x) = \pi_p(x) \Phi_s$ for all $x \in \aA$.
	\end{enumerate}
\end{lemma}
\begin{proof}
	For the first statement we claim that $s(\mathfrak{K}_\aA) = \hH_p$. 
	From this it follows that $\Phi_s$ is well-defined, surjective, and hence unitary, because $s$ is an isometry. 
	To prove the claim, we apply $ss^*s = s$ and $s$ to the inclusions $s(\mathfrak{K}_\aA) \subseteq \hH_\aA(1) \tensor \hH$ and \mbox{$s^*(\hH_\aA(1) \tensor \hH) \subseteq \mathfrak{K}_\aA$}, respectively, which in turn are a consequence of Equation~\eqref{eq:SOPequivariance}. 
	Furthermore, the commutation relations $\lambda_g r_g = r_g \lambda_g$, $g \in G$, and Equation~\eqref{eq:SOPcommuting} make it obvious that $\Phi_s$ is $G$-equivariant.
	The second statement is due to the fact that $s$ is normalized, while the last statement is clear from the defining Equation~\eqref{eq:pi_S}. 
\end{proof}

\begin{theorem}\label{thm:class_cov_rep}
	Let $(\aA, G, \alpha)$ be a free C\Star-dynamical system with fixed point algebra $\aB$. Furthermore, let $\hH_\aB$ be a nontrivial representation space of $\aB$. Then each faithful covariant representation of $(\aA, G, \alpha)$ such that the trivial isotypic component of the underlying representation space is isomorphic to $\hH_\aB$ is, up to unitary equivalence, of the form 
	\begin{align}
		\bigl(p(\hH_\aB \tensor \hH),\pi_p,\one_\aB \tensor \mu \bigr),\label{eq:class_cov_rep}
	\end{align}
	where $\hH$, $\mu$, $p$, and $\pi_p$ are as above (\cf Lemma~\ref{lem:coisometry} and Equation~\eqref{eq:pi_p}).
\end{theorem}
\begin{proof}
	We shall have established the theorem if we show that there is a unitary map $\Phi : \hH_\aA \to \hH_p$ such that $\Phi u_g = (\one_\aA \tensor \mu_g) \Phi$ for all $g \in G$, $\Phi(\eta) = \eta \tensor 1_\C$ for all $\eta \in \hH_\aA(\one)$, and $\Phi \pi(x) = \pi_p(x) \Phi$ for all $x \in  \aA$. 
	For this, we simply need to look at the composition $\Phi_s J_u $, because the map $J_u$ happens to satisfy equations similar to those of Lemma~\ref{lem:Sunitary}.
\end{proof}

Taking advantage of Remark~\ref{rem:partial_iso}, we even get the following result:

\begin{corollary}
	All faithful covariant representations of $(\aA, G, \alpha)$ are unitarily equivalent.
\end{corollary}

We continue by considering an arbitrary covariant representation $(\pi,u)$ of $(\aA, G,\alpha)$ on a Hilbert space $\hH_\aA$ and bring to mind that $\pi\bigl( A(\sigma) \bigr) \acts \hH_\aA(1) \subseteq \hH_\aA(\sigma)$ for all $\sigma \in \Irrep(G)$. 
Our second goal is to establish that these inclusions are, in fact, equalities.
To prove this, we proceed as follows. 
For each $\sigma \in \Irrep(G)$ we identify
\begin{equation*}
	\Gamma_{\hH_\aA}(\sigma) := \{\eta \in \hH_\aA \tensor V_\sigma : u_g \tensor \sigma_g(\eta) = \eta \quad \forall g \in G \}
\end{equation*}
with the multiplicity space of $\hH_\aA(\bar{\sigma})$ and note that $\Phi_u \tensor \id_{V_\sigma}$ provides an isomorphism between $\Gamma_{\hH_\aA}(\sigma)$ an the multiplicity space of $\mathfrak{K}_\aA(\bar{\sigma})$, the latter being the $\bar{\sigma}$-isotypic component of $\mathfrak{K}_\aA$ under the action $\one_\aA \tensor \lambda_g$, $g\in G$. 
Furthermore, for each $\sigma \in \Irrep(G)$ we write $\hH_\sigma$ for the multiplicity space of $\hH(\bar{\sigma})$ (\cf~Equation~\eqref{eq:decompH}) and infer that $p(\hH_\aA(1) \tensor \hH_\sigma)$ may be regarded as the multiplicity space of $\hH_p(\bar{\sigma}) = p(\hH_\aA(1) \tensor \hH)(\bar{\sigma})$. 
From this and Lemma~\ref{lem:Sunitary} it follows that for each $\sigma \in \Irrep(G)$ the map
\begin{equation}
	\hH_\aA(1) \tensor \hH_{\sigma} \to \Gamma_{\hH_\aA}(\sigma), 
	\qquad 
	\zeta \mapsto \pi\bigl( s(\sigma) \bigr)^*(\zeta)\label{eq:sbar}
\end{equation}
is surjective, where $\pi\bigl( s(\sigma) \bigr) \in \pi(\aA) \tensor \End(V_\sigma,\hH_\sigma) \subseteq \End(\hH_\aA \tensor V_\sigma,\hH_\aA \tensor \hH_\sigma)$ denotes the isometry coming from disassembling $s$ (\cf~Lemma~\ref{lem:coisometry}). 
Finally, we recall that for each $\sigma \in \Irrep(G)$ the map $\Gamma_{\hH_\aA} (\bar{\sigma}) \tensor V_\sigma \to \hH_\aA(\sigma)$ given on simple tensors by $\zeta \tensor v \mapsto \id \tensor \text{ev}_v(\zeta)$ is unitary, where $\text{ev}_v(\bar{w}):=  \rprod{}{w,v}$ for all $\bar{w} \in \bar{V}_\sigma$. 
Having disposed of these preliminary steps, for each $\sigma \in \Irrep(G)$ we may now conclude that 
\begin{align*}
	\hH_\aA(\sigma) 
	&= \lin \{ \ev_v(\zeta): v \in V_\sigma, \zeta \in \Gamma_{\hH_\aA}(\bar \sigma) \}
	\\
	&= \lin \bigl\{ \ev_v\bigl( \pi\bigl( s(\bar \sigma) \bigr)^*(\xi \tensor \eta) \bigr): v \in V_\sigma, \xi \in \hH_\aA(1), \eta \in \hH_{\bar \sigma} \bigr\}.
\end{align*}
Since for each $v \in V_\sigma$ and $\eta \in \hH_{\bar \sigma}$ the operator $\hH_\aA(\one) \ni \xi \mapsto \ev_v\bigl( \pi\bigl( s(\bar \sigma) \bigr)^*(\xi \tensor \eta) \bigr) \in \hH_\aA$ lies in $\pi\bigl(A(\sigma)\bigr)$, we have established our claim:

\begin{corollary}\label{cor:Sunitary}
	Let $(\pi,u)$ be a covariant representation of $(\aA, G,\alpha)$ on a Hilbert space~$\hH_\aA$. Then for each $\sigma \in \Irrep(G)$ we have $\pi\bigl( A(\sigma) \bigr) \acts \hH_\aA(1) = \hH_\aA(\sigma)$.
\end{corollary}

\pagebreak[3]
\section{Lifting spectral triples}\label{sec:permanence}

In this section we study permanence properties of free C\Star-dynamical systems with respect to spectral triples. 
For a start we fix the following data:
\begin{itemize}
	\item
		a compact Lie group $G$ of dimension $n$ with Lie algebra $\Lie(G)$;
	\item
		a free C\Star-dynamical system $(\aA,G,\alpha)$ with fixed point algebra~$\aB$;
	\item
		for each $\sigma \in \Irrep(G)$ a finite-dimensional Hilbert space $\hH_\sigma$ and an isometry $s(\sigma)$ in  $\aA \tensor \End(V_\sigma,\hH_\sigma)$ satisfying $\alpha_g\bigl(s(\sigma)\bigr)=s(\sigma) (\one_\aA \tensor \sigma_g)$ for all $g \in G$ (\cf~\cite[Lem.~3.2]{SchWa17}). 
		In particular, for $1 \in \Irrep(G)$, we choose $\hH_1 := \C$ and $s(1) := \one_\aA$; 
	\item
		a spectral triple $\DD_\aB := (\aB_0,\pi_\aB,\hH_\aB,D_\aB)$ on $\aB$.
\end{itemize}
Our main objective is to construct a spectral triple $\mathbb{D}_\aA$ on $\aA$ by means of the above data that extends $\DD_\aB$ in the sense of Definition~\ref{def:lift} and incorporates the geometry of $G$.

\begin{defn}\label{def:lift}
	Let $\DD_\aA= (\aA_0,\pi_\aA,\hH_\aA,D_\aA)$ and $\DD_\aB = (\aB_0,\pi_\aB, \hH_\aB, D_\aB)$ be two spectral triples and suppose that $\aB_0 \subseteq \aA_0$. We say that $\DD_\aA$ is a \emph{lift} of $\DD_\aB$ or that $\DD_\aA$ \emph{lifts} $\DD_\aB$ if there is an isometry $t: \hH_\aB \to \hH_\aA$ such that the following conditions are satisfied: 
	\begin{enumerate}[label={\arabic*}.,ref=\ref{def:lift}.{\textit{\arabic*}}.]
		\item\label{def:liftrep}
			$\pi_\aA(b) t = t \pi_\aB(b)$ for all $b\in \aB_0$;
		\item\label{def:liftD}
			$t \bigl( \dom(D_B) \bigr) \subseteq \dom(D_A)$ and $D_A t = t D_B $ on $\dom(D_B)$.
	\end{enumerate}
	Combining~\ref{def:liftrep} and~\ref{def:liftD} gives $[D_\aA, \pi_\aA(b)] t = t [D_\aB, \pi_\aB(b)]$ for all $b\in \aB$. 
\end{defn}

\pagebreak[3]
\begin{remark}\label{rmk:strong_lift}
	Definition~\ref{def:lift} is weak in the sense that it only addresses the Hilbert space $t(\hH_\aB) \subseteq \hH_\aA$. 
	We may look at the sets of differential 1-forms $\Omega^1(\aA_0)$ and $\Omega^1(\aB_0)$ associated with the spectral triples $\DD_\aA$ and $\DD_\aB$, respectively (see, \eg,~\cite[Sec.~8.1]{GaPrRa00} and ref.~therein). 
	If $\DD_\aA$ is a lift of $\DD_\aB$, then it follows immediately that
	\begin{equation*}
		\pi_\aA(x) [D_\aA, \pi_\aA(y)] \, t = t \, \pi_\aB(x) [D_\aB, \pi_\aB(y)]
	\end{equation*}
	for all $x,y \in \aB_0$, and hence $\Omega^1(\aB_0) = t^* \Omega^1(\aA_0) t$. 
	However, in general $\Omega^1(\aB_0)$ does not embed into $\Omega^1(\aA_0)$ as a $\aB_0$-bimodule. 
	For the na\"\i ve construction of a lift given in this article the natural embedding $\pi_\aB(x) [D_\aB,\pi_\aB(y)] \mapsto \pi_\aA(x) [D_\aA, \pi_\aA(y)]$ for all $x,y \in \aB_0$ is well-defined only under additional assumptions. 
	Without further requirements we do not yet know whether this map becomes well-defined with a more refined construction. 
\end{remark}

\begin{example}
	In~\cite{DaSi13,DaZu13} the authors consider spectral triples that are equivariant with respect to a torus action. Given such a spectral triple $\DD_\aA = (\aA_0, \pi_\aA, \hH_\aA, D_\aA, J_\aA)$, they show that $\DD_\aA$ may under certain conditions (\cf~\cite[Def. 3.3 and Def. 3.4]{DaZu13}) be restricted to a spectral triple $\DD_\aB = (\aB_0, \pi_\aB, \hH_\aB, D_\aB, J_\aB)$ of the respective fixed point algebra $\aB_0$. Then it is an easy matter to check that $\mathbb D_\aA$ lifts $\mathbb D_\aB$ in the sense of Definition~\ref{def:lift}.
\end{example}

A key feature of our free C\Star-dynamical system $(\aA,G,\alpha)$ is the factor system associated with the isometries $s(\sigma)$, $\sigma \in \Irrep(G)$, (see~\cite[Def.~4.1]{SchWa17}), which we now recall for the convenience of the reader. 
Given a finite-dimensional representation $(\sigma,V_\sigma)$ of $G$, we decompose it into irreducible subrepresentations $\sigma = \sigma_1 \oplus \dots \oplus \sigma_n$ and define an isometry $s(\sigma) \in \aA \tensor \End(V_\sigma,\hH_\sigma)$ satisfying $\alpha_g \bigl(s(\sigma) \bigr) = s(\sigma) (\one_\aA \tensor \sigma_g)$ for all $g \in G$ by summing up the isometries $s(\sigma_1), \dots, s(\sigma_n)$. 
That is, we put $\hH_\sigma := \hH_{\sigma_1} \oplus \dots \oplus \hH_{\sigma_n}$ and $s(\sigma) := s(\sigma_1) \oplus \dots \oplus s(\sigma_n)$. 
In this way we extend the mapping $\sigma \mapsto \hH_\sigma$ to an additive functor from the representation category of $G$ into the category of finite-dimensional Hilbert spaces and the mapping $\sigma \mapsto s(\sigma)$ to a family of $G$-equivariant isometries that is indexed by the representation category of $G$ and behaves naturally with respect to intertwiners. 
Notably, it is easy to check that for each finite-dimensional representation $\sigma$ of $G$ we have 
\begin{align*}
	\Gamma_\aA(\sigma) = s(\sigma)^*( \aB \tensor \hH_\sigma).
\end{align*}
Furthermore, we obtain a \Star-homomorphism 
\begin{equation*}
	\gamma_\sigma: \aB \to \aB \otimes \End(\hH_\sigma),
	\quad
	\gamma_\sigma(b) := s(\sigma) (b \tensor \one_{V_\sigma}) s(\sigma)^*,
\end{equation*}
to which we refer as \emph{coaction} of the factor system. 
For each pair $(\sigma,\tau)$ of finite-dimensional representations of $G$ we obtain an element 
\begin{equation*}
	\omega(\sigma, \tau) := s(\sigma \tensor \tau) s(\sigma)^* s(\tau)^* \in \aB \tensor \End(\hH_\sigma \tensor \hH_\tau,\hH_{\sigma \tensor \tau}),
\end{equation*}
which we call the \emph{cocycle} of the factor system. 
Here $s(\sigma)$ and $s(\tau)$ are regarded amplified to act trivially on $V_\tau$ and $\hH_\sigma$, respectively. 
The most important relations of the coactions and the cocycles are captured by the following equations:
\begin{align}
	\omega(\sigma,\tau) \omega(\sigma,\tau)^* = 	\gamma_{\sigma \tensor \tau}(\one_\aB), & \qquad \omega(\sigma,\tau)^* \omega(\sigma,\tau) = (\gamma_\tau)_{13} \bigl( \gamma_\sigma(\one_\aB) \bigr), 
	\label{eq:ranges_sys}
	\\
	 \gamma_{\sigma \tensor \tau}(b) \omega(\sigma, \tau) &= \omega(\sigma, \tau) (\gamma_\tau)_{13} \bigl( \gamma_\sigma(b) \bigr), 
	\notag
	\\ 
	\omega(\sigma, \tau \tensor \rho) \omega(\tau, \rho)_{134} &= \omega(\sigma \tensor \tau, \rho) (\gamma_\rho)_{14} \bigr( \omega(\sigma, \tau) \bigr) 
	\notag\label{eq:cocycle_sys}
\end{align}
for all finite-dimensional representations $\sigma, \tau, \rho$ of $G$ and $b \in \aB$ (see~\cite[Lem.~4.3]{SchWa18}). The subindices refer to the leg numbering within the underlying tensor product.

\subsection{Lifting the algebra}\label{sec:liftA}

We do note require that $\aA$ comes equipped with a dense unital \Star-subalgebra of smooth elements as initial data. 
Hence, as a fist step towards a spectral triple on $\aA$, we fix such an algebra. 
To this end, we recall that  each isotypic component $\aA(\bar\sigma)$, $\sigma \in \Irrep(G)$, admits a linear bijection $\Phi_\sigma: \Gamma_\aA(\sigma) \tensor \bar V_\sigma \to \aA(\bar \sigma)$ (see~Equation~\eqref{eq:iso_section}). 
Since $\Gamma_\aA(\sigma) = s(\sigma)^*(\aB \tensor \hH_\sigma)$, it immediately follows that $\aA(\bar\sigma)$ is linearly spanned by the elements
\begin{equation}\label{eq:element}
	a_\sigma(b \tensor \eta \tensor \bar v) := \Phi_\sigma \bigl( s(\sigma)^*(b \tensor \eta)  \tensor \bar v \bigr),
\end{equation}
where $b \in \aB$, $\eta \in \hH_\sigma$, and $\bar v \in \bar V_\sigma$. We may extend this notation linearly to define an element $a_\sigma(x) \in \aA(\bar\sigma)$ for each $x \in \aB \tensor \hH_\sigma \tensor \bar V_\sigma$. Moreover, we may extend Equation~\eqref{eq:element} to any finite-dimensional representation $\sigma$ of $G$ to get an element $a_\sigma(x) \in \aA$ for each $x \in \aB \tensor \hH_\sigma \tensor \bar V_\sigma$.
The action $\alpha$ and the multiplication on $\aA$ then take the form
\begin{align}
	\label{eq:action_on_elements}
	\alpha_g \bigl( a_\sigma(b \tensor \eta \tensor \bar v) \bigr)
	&= a_\sigma(b \tensor \eta \tensor \bar\sigma_g v),
	\\
	\label{eq:multiplication_of_elements}
	a_\sigma(b \tensor \eta \tensor \bar v) \cdot a_\tau(c \tensor \vartheta \tensor \bar w)
	&= a_{\sigma \tensor \tau} \bigl( \omega(\sigma, \tau) \; \gamma_\tau(b)_{13} \; (\xi \tensor \eta \tensor \vartheta \tensor \bar v \tensor \bar w) \bigr)
\end{align}
for all $g \in G$, $\sigma, \tau \in \Irrep(G)$, $b,c \in \aB$, $\eta \in \hH_\sigma$, $\vartheta \in \hH_\tau$, $\bar v \in \bar V_\sigma$, and $\bar w \in \bar V_\tau$. 
It is shown in~\cite[Sec.~5]{SchWa17} that also the involution can be made explicit on these elements. 
Indeed, for each $\sigma \in \Irrep(G)$ there is an antilinear map $J_\sigma: \aB \tensor \hH_\sigma \to \aB \tensor \hH_{\bar \sigma}$ such that
\begin{equation} 
	\label{eq:involution_of_elements}
	a_\sigma(x \tensor \bar v)^* = a_{\bar \sigma} \bigl( J_\sigma(x) \tensor v \bigr)
\end{equation}
for all $x \in \aB \tensor \hH_\sigma$ and $\bar v \in \bar V_\sigma$. Summarizing, we can assert that the elements $a_\sigma(x)$ for a finite-dimensional representation $\sigma$ of $G$ and $x \in \aB \tensor \hH_\sigma \tensor \bar V_\sigma$ form a unital \Star-subalgebra, which is dense and $G$\nobreakdash-invariant. Furthermore, under some relatively mild conditions on the factor system we may restrict to elements $x \in \aB_0 \tensor \hH_\sigma \tensor \bar V_\sigma$:

\begin{theorem}\label{thm:A_0}
	Suppose that for each $\sigma, \tau \in \Irrep(G)$ we have $\gamma_\sigma(\aB_0) \subseteq \aB_0 \tensor \End(\hH_\sigma)$ and $\omega(\sigma, \tau) \in \aB_0 \tensor \End(\hH_\sigma \tensor \hH_\tau, \hH_{\sigma \tensor \tau})$. 
	Then the set 
	\begin{equation*}
		\aA_0 := \{ a_\sigma(x)  : \sigma~\text{finite-dimensional representation of $G$}, x \in \aB_0 \tensor \hH_\sigma \tensor \bar V_\sigma \}.
	\end{equation*}
	is a dense and $G$-invariant unital \Star-subalgebra of $\aA$ satisfying $\aA_0^G = \aB_0$. 
	The action~$\alpha$, the multiplication, and the involution on $\aA_0$ are given by the Equations~\eqref{eq:action_on_elements}, \eqref{eq:multiplication_of_elements}, \eqref{eq:involution_of_elements}, respectively. 
\end{theorem}
\begin{proof}
	Choosing $\sigma$ trivial, we see at once that $\aB_0 \subseteq \aA_0$. 
	Since $\aB_0$ is dense in $\aB$, it follows that $\aA_0$ is dense in $\aA$. 
	Equation~\eqref{eq:action_on_elements} shows that $\aA_0$ is $\alpha$-invariant and that $\aA_0^G = \aB_0$. 
	Equation~\eqref{eq:multiplication_of_elements} implies that $\aA_0$ is, under the given conditions, a subalgebra of $\aA$. 
	We leave it to the reader to follow the construction in~\cite[Sec.~5]{SchWa17} to verify that $J_\sigma(\aB_0 \tensor \hH_\sigma) \subseteq \aB_0 \tensor \hH_{\bar \sigma}$ for all $\sigma \in \Irrep(G)$ and hence that $\aA_0$ is, in fact, a \Star-algebra.
\end{proof}

Throughout the rest of the paper we make the standing assumptions that 
\begin{align*}
	\gamma_\sigma(\aB_0) &\subseteq \aB_0 \tensor \End(\hH_\sigma)
	&
	&\text{and}
	&
	\omega(\sigma, \tau) &\in \aB_0 \tensor \End(\hH_\sigma \tensor \hH_\tau, \hH_{\sigma \tensor \tau})
\end{align*}
for all $\sigma, \tau \in \Irrep(G)$, and we proceed with the subalgebra $\aA_0$ as the algebra of smooth functions in the spectral triple on $\aA$. 

\begin{remark}\label{rem:A_0}
	\begin{enumerate}
		\item 
			Since the action $\alpha$ is smooth on each isotypic component, the algebra~$\aA_0$ is contained in $\aA^\infty$, the smooth domain of~$\alpha$.
		\item
			Putting $A_0(\sigma) := \Span \{ a_\sigma(x) : x \in \aB_0 \tensor \hH_{\bar{\sigma}} \tensor V_\sigma \}$  for each $\sigma \in \Irrep(G)$, we see at once that $\aA_0$ decomposes into the algebraic direct sum $\bigoplus_{\sigma\in\Irrep(G)}^{\text{alg}}A_0(\sigma)$.
	\end{enumerate}
\end{remark}

\begin{remark}
	Concrete examples of C\Star-dynamical systems typically offer a canonical subalgebra for a spectral triple. 
	For a reasonable choice of isometries $s(\sigma)$, $\sigma \in \Irrep(G)$, the algebra $\aA_0$ is included, because it is the minimal \Star-subalgebra compatible with the given data.
	More precisely, $\aA_0$ is the smallest \Star-subalgebra of $\aA$ that is invariant under~$\alpha$, satisfies $\aB_0 \subseteq \aA_0^G$, and such that each $s(\sigma)$, $\sigma \in \Irrep(G)$, lies in $\aA_0 \tensor \End(V_\sigma, \hH_\sigma)$. 
\end{remark}

\subsection{Lifting the \texorpdfstring{$^*$}{*}-representation}\label{sec:liftrep}

Our next goal is to provide a faithful covariant representation of $(\aA,G,\alpha)$ that extends the $^*$-representation $\pi_\aB : \aB \to \End(\hH_\aB)$. For this purpose, we consider the unitary representation $\mu: G \to \mathcal U(\hH)$ on $\hH$ as introduced in the proof of Lemma~\ref{lem:coisometry}, that is,
\begin{align}\label{eq:decH}
	\hH &:= \bigoplus_{\sigma \in \Irrep(G)} \hH_\sigma \tensor \bar{V}_\sigma
		&
		&\text{and}
		&
	\mu_g &:= \bigoplus_{\sigma \in \Irrep(G)} \one_{\hH_\sigma} \tensor \bar{\sigma}_g.
\end{align}
We recall from Corollary~\ref{cor:pi_s} that $\aA$ admits a faithful \Star-homomorphism 
\begin{align*}
	\pi_s: \aA \to \Mul \bigl( ss^* \bigl(\aB \tensor \Com(\hH)  \bigr) ss^* \bigr).
\end{align*}
Composing this with the faithful \Star-representation $\pi_\aB \tensor \id : \aB \tensor \Com(\hH) \to \End(\hH_\aB \tensor \hH)$ and putting $p := \pi_\aB \tensor \id(ss^*) \in \End(\hH_\aB \tensor \hH)$, we obtain a faithful \Star-representation of~$\aA$ on the Hilbert space $\hH_p := p (\hH_\aB \tensor \hH)$:
\begin{equation}\label{eq:hatpi}
	\pi_\aA : \aA \to \End(\hH_p),
	\qquad
	\pi_\aA:= (\pi_\aB \tensor \id) \circ \pi_s.
\end{equation}
For an explicit form of $\pi_\aA$, we proceed analogously to Section~\ref{sec:liftA}. 
To simplify notation we regard $\aB$ as subalgebra of $\End(\hH_\aB)$, omitting the representation $\pi_\aB$. 
For $\sigma \in \Irrep(G)$, $\xi \in \hH_\aB$, $\eta \in \hH_\sigma$, and $\bar v \in \bar V_\sigma$ we define a vector in $\hH_p$ by
\begin{equation*}
	\psi_\sigma(\xi \tensor \eta \tensor \bar v) := s(\sigma) s(\sigma)^* (\xi \tensor \eta) \tensor \bar v.
\end{equation*}
We extend this notation linearly in all components to arbitrary finite-dimensional representations $\sigma$ of~$G$ and vectors $x \in \hH_\aB \tensor \hH_\sigma \tensor \bar V_\sigma$ as argument and notice that the vectors $\psi_\sigma(x)$ are dense in $\hH_p$. The action $u_\aA := \one_\aB \tensor \mu$ on $\hH_p$ and the \Star-representation $\pi_\aA$ then take the form
\begin{align}
	\notag
	(u_\aA)_g \acts \psi_\sigma(\xi \tensor \eta \tensor \bar v) 
	&= \psi_\sigma(\xi \tensor \eta \tensor \bar \sigma_g \bar v), 
	\\
	\label{eq:multiplication_with_vectors}
	\pi_\aA \bigl( a_\sigma(b \tensor \eta \tensor \bar v) \bigr) \acts \psi_\tau(\xi \tensor \vartheta \tensor \bar w)
	&= \psi_{\sigma \tensor \tau} \bigl( \omega(\sigma, \tau) \; \gamma_\tau(b)_{13} \; (\xi \tensor \eta \tensor \vartheta \tensor \bar v \tensor \bar w) \bigr)
\end{align}
for all $g \in G$, $\sigma, \tau \in \Irrep(G)$, $b \in \aB$, $\xi \in \hH_\aB$, $\eta \in \hH_\sigma$, $\vartheta \in \hH_\tau$, $\bar v \in \bar V_\sigma$, and $\bar w \in \bar V_\tau$. 
One instantly becomes aware of the similarities to the Equations~\eqref{eq:action_on_elements} and~\eqref{eq:multiplication_of_elements} above. 
Choosing the trivial representation as $\sigma$, we see at once that $\pi_\aA(b)$ acts as $\gamma_\tau(b)$ for all $b \in \aB$, and, in consequence, the \Star-representation $\pi_\aB$ is recovered on $\hH_\aB = \hH_\aB \tensor \C \subseteq \hH_p$. 
That is, in summary:

\begin{theorem}\label{thm:lift_cov_rep}
	The pair $(\pi_\aA,u_\aA)$ is a faithful covariant representation of $(\aA, G, \alpha)$ on~$\hH_p$. Furthermore, writing $t:\hH_\aB \to \hH_p$ for the isometry given by $t(\xi) := \xi \tensor 1_\C$, we have $\pi_\aA(b) t = t \pi_\aB(b)$ for all $b \in \aB$.
\end{theorem}

\subsection{Lifting the Dirac operator}\label{sec:liftdirac}

We now turn to the construction of a Dirac operator, which is the only point remaining in our endeavour to establish a spectral triple  on $\aA$. 
The procedure naturally falls into three parts. 
First, we construct a ``horizontal'' lift of $D_\aB$ to an operator on~$\hH_p$. 
Second, we associate a ``vertical'' Dirac operator with the unitary representation $\mu: G \to \mathcal U(\hH)$. 
Finally, we put together the horizontal and the vertical part in a suitable way.

\subsubsection{The horizontal lift}\label{sec:horlift}
	
By the definition of the \Star-algebra $\aA_0$, for each $\sigma \in \Irrep(G)$ the operator $p(\sigma) := s(\sigma) s(\sigma)^*$ lies in $\aB_0 \tensor \End(\hH_\sigma)$. Hence Lemma~\ref{lem:resD} implies that the unbounded operator 
\begin{equation*}
	D_\sigma := p(\sigma) ( D_\aB \tensor \one_{\hH_\sigma} ) p(\sigma)
	\qquad
	\text{with domain}~\dom(D_\sigma) := p(\sigma) ( \dom(D_\aB) \tensor \hH_\sigma)
\end{equation*}
on the Hilbert space $\hH_{p(\sigma)} = p(\sigma)(\hH_\aB \tensor \hH_\sigma)$ is self-adjoint and has compact resolvent. In particular, we have $D_\one = D_\aB$. Passing over to the Hilbert space direct sum, we may conclude, from~\cite[Lem.~5.3.7]{Ped89} for instance, that there is a unique self-adjoint operator, let's say, $D_h$ on $\hH_p$ such that  $D_h \mid \dom( D_\sigma) \tensor \bar{V}_\sigma = D_\sigma \tensor \one_{\bar{V}_\sigma}$ for all $\sigma \in \Irrep(G)$. We now put all of this on record:

\begin{corollary}\label{cor:D_h}
	The following assertions hold for the unbounded operator $D_h$ on $\hH_p$:
	\begin{enumerate}[label={\arabic*}.,ref=\ref{lcor:D_h}.{\textit{\arabic*}}.]
	\item
		$D_h$ is a self-adjoint.\label{D_hsa}
	\item
		$D_h t = t D_\aB$ on $\dom(D)$, where $t:\hH_\aB \to \hH_p$ denotes the isometry from Theorem~\ref{thm:lift_cov_rep}.\label{D_hres}
	\item
		$(\one_\aB \tensor \mu_g)  D_h = D_h (\one_\aB \tensor \mu_g) $ for all $g \in G$.\label{D_mu}
	\end{enumerate}
\end{corollary}

The task is now to establish that $D_h$ has bounded commutators with $\aA_0$. 
For a start we notice that the span of vectors $\psi_\sigma(\xi \tensor \eta \tensor \bar v)$ for a finite-dimensional representation $\sigma$ of $G$, $\xi \in \dom(D_\aB)$, $\eta \in \hH_\sigma$, and $\bar v \in \bar V_\sigma$ lie in the domain of $D_h$ and that
\begin{equation}
	\label{eq:Dirac_on_vectors}
	D_h \psi_\sigma(\xi \tensor \eta \tensor \bar v)
	= \psi_\sigma \bigl( (D_\aB^{})_1 \, p(\sigma)_{12} \, (\xi \tensor \eta \tensor \bar v) \bigr),
\end{equation}
where the index refers to the leg numbering in $\hH_\aB \tensor \hH_\sigma \tensor \bar V_\sigma$. 
Moreover, by our standing assumptions, we have
\begin{gather}
	[D_\aB \tensor \one_{\hH_\sigma}, \gamma_\sigma(b)] \in \End(\hH_\aB \tensor \hH_\sigma) \qquad \forall b \in \aB_0,
	\label{eq:commD_sigma}
	\\
	 [D_\aB \tensor \one, \omega(\sigma,\tau)] \in \End(\hH_\aB \tensor \hH_\sigma \tensor \hH_\tau, \hH_\aB \tensor \hH_{\sigma \tensor \tau})
	 \label{eq:commD_sigma,pi}
\end{gather}
for all finite-dimensional representations $\sigma, \tau$ of $G$, where $[D_\aB \tensor \one, \omega(\sigma,\tau)]$ informally stands for the difference $(D_\aB \tensor \one_{\hH_{\sigma \tensor \tau}}) \omega(\sigma,\tau) - \omega(\sigma,\tau) (D_\aB \tensor \one_{\hH_{\sigma}} \tensor \one_{\hH_{\tau}})$. 

\begin{theorem}\label{thm:D_hcomm}
	Suppose that there are constants $C(b)>0$ and $C(\sigma)>0$ for all $b \in \aB_0$ and $\sigma \in \Irrep(G)$, respectively, such that
	\begin{gather*}
		\sup_{\tau \in \Irrep(G)} \norm[\big]{ [D_\aB \tensor \one_{\hH_\tau}, \gamma_\tau(b)] } < C(b) \qquad \forall b \in \aB_0,		
		\\
	\sup_{\tau \in \Irrep(G)} \norm[\big]{ [D_\aB \tensor \one, \omega(\sigma,\tau)] } < C(\sigma) \qquad \forall \sigma \in \Irrep(G).
	\end{gather*}
	Then $D_h$ has bounded commutators with $\aA_0$.
\end{theorem}
\begin{proof}
	It clearly suffices to demonstrate that $D_h$ has a bounded commutator with each operator $\pi_\aA \bigl( a_\sigma(x) \bigr)$ for a finite-dimensional representation $\sigma$ of $G$ and $x = b \tensor \eta \tensor \bar v$ with $b \in \aB_0$, $\eta \in \hH_\sigma$, and $\bar v \in \bar V_\sigma$. 
	Hence let us consider such an operator and, in addition, a vector $\psi_\tau(y)$ for some finite-dimensional representation $\tau$ of $G$ and 
	\begin{align*}
		y = \sum_{i \in I} \xi_i \tensor \zeta_i \tensor \bar w_i \in \dom(D_\aB) \tensor \hH_\tau \tensor \bar V_\tau.
	\end{align*}
	Then the Equations~\eqref{eq:multiplication_with_vectors} and~\eqref{eq:Dirac_on_vectors} together with the identity $p(\sigma \tensor \tau) \omega(\sigma, \tau) = \omega(\sigma, \tau)$ from Equation~\eqref{eq:ranges_sys} imply
	\begin{align*}
		D_h \pi_\aA \bigl( a_\sigma(x) \bigr) \psi_\tau(y) 
		&= \sum_{i \in I} \psi_{\sigma \tensor \tau} \bigl( (D_\aB)_1 \, \omega(\sigma, \tau)_{123} \, \gamma_\tau(b)_{13} \, (\xi_i \tensor \eta \tensor \zeta_i \tensor \bar v \tensor \bar w_i) \bigr)
		\shortintertext{and}
		\pi_\aA \bigl( a_\sigma(x) \bigr) D_h \psi_\tau(y) 
		&= \sum_{i \in I} \psi_{\sigma \tensor \tau} \bigl( \omega(\sigma,\tau)_{123} \, \gamma_\tau(b)_{13} \, (D_\aB)_1 \, p(\tau)_{13} \, (\xi_i \tensor \eta \tensor \zeta_i \tensor \bar v \tensor \bar w_i) \bigr).
	\end{align*}
	For the commutator it follows that
	\begin{equation*}
		\bigl[ D_h, \pi_\aA \bigl( a_\sigma(x) \bigr) \bigr] \psi_\tau(y) 
		= \sum_{i \in I} \psi_{\sigma \tensor \tau} \bigl( [ D_\aB \tensor \one, \omega(\sigma, \tau) \gamma_\tau(b)_{13}] \, p(\tau)_{13} \, (\xi_i \tensor \eta \tensor \zeta_i \tensor \bar v \tensor \bar w_i)  \bigr).
	\end{equation*}
	By the hypothesis, the operator $[D_\aB \tensor \one, \omega(\sigma, \tau) \gamma_\tau(b)_{13}]$ is bounded in~$\tau$ by some constant that only depends on $\sigma$ and $b$, let's say $C(\sigma,b)$, and, in consequence, 
	\begin{align*}
		\norm[\big]{ \bigl[ D_h, \pi_\aA \bigl( a_\sigma(x) \bigr) \bigr] \psi_\tau(y) } &\le \norm[\big]{ [D_\aB \tensor \one, \omega(\sigma,\tau) \gamma_\tau(b)_{13}] } 
		\cdot \norm[\Big]{ \sum_{i \in I} p(\tau)_{13} (\xi_i \tensor \eta \tensor \zeta_i \tensor \bar v \tensor \bar w_i) }
		\\
		&\le C(\sigma,b) \cdot \norm{\eta \tensor \bar v} \cdot \norm{\psi_\tau(y)}.
	\end{align*}
	Therefore, $[D_h, \pi_\aA \bigl( a_\sigma(x) \bigr)]$ is bounded, because the vectors $\psi_\tau(y)$ for a finite-dimensional representation $\tau$ of $G$ and $y \in \dom(D) \tensor \hH_\tau \tensor \bar V_\tau$ are dense in~$\hH_p$. 
\end{proof}

\begin{remark}
	The condition that the operators in Equation~\eqref{eq:commD_sigma} are uniformly bounded is closely related to the notion of a so-called  equicontinuous group action on a spectral metric space, \cite[Def.~3]{BeMaRe10} (see Section~\ref{sec:crossprod} below).
\end{remark}
 
\begin{remark}
	The aim of this remark is to show that the commutators in Equation~\eqref{eq:commD_sigma,pi} do not vanish in general. 
	To this end, we recall a C\Star-algebraic version of the nontrivial Hopf-Galois extensions studied in~\cite{LaSu05} (see also~\cite{CoLa01}). 
	Let $\theta \in \R$ and let $\theta'$ denote the skewsymmetric $4 \times 4$-matrix with $\theta_{1,2}' = \theta_{3,4}' = 0$ and $\theta'_{1,3} = \theta'_{1,4} = \theta_{2,3}' = \theta'_{2,4} = \theta/2$. 
	The Connes-Landi sphere $\aA(\mathbb S_{\theta'}^7)$ is the universal unital C\Star-algebra generated by normal elements $z_1, \dots, z_4$ subject to the relations
	\begin{align*}	
		z_i z_j &= e^{2\pi\imath \theta'_{i,j}} \; z_j z_i, 
		&
		z_j^* z_i &= e^{2\pi \imath \theta'_{i,j}}\;  z_i z_j^*,
		&
		\sum_{k=1}^4 z_k^* z_k^{} &= \one
	\end{align*}
	for all $1 \le i,j \le 4$. 
	On account of \cite[Expl.~3.5]{SchWa17}, it comes equipped with a free action of the non-Abelian group $G =\SU(2)$ given for each $U \in \SU(2)$ on generators by
	\begin{equation*}
		\alpha_U: (z_1, \dots, z_4) \mapsto (z_1, \dots, z_4) \begin{pmatrix} U & 0 \\ 0 & U \end{pmatrix}.
	\end{equation*}
	The corresponding fixed point algebra is the universal unital C\Star-algebra $\aA(\mathbb S_\theta^4)$ generated by normal elements $w_1, w_2$ and a self-adjoint element $x$ satisfying
	\begin{align*}
		w_1 w_2 &= e^{2\pi\imath\theta} \; w_2 w_1, 
		&
		w_2^* w_1 &= e^{2\pi\imath \theta} \; w_1 w_2^*, 
		&
		w_1^* w_1 + w_2^* w_2 + x^*x &= \one.
	\end{align*}
	Moreover, for the fundamental representation $(\sigma_1,\mathbb{C}^2)$ of $\SU(2)$ the element
	\begin{equation*}
		s(\sigma_1) := \begin{pmatrix}
			z_1^* & z_2^*
			\\
			-z_2 & z_1
			\\
			z_3^* & z_4^*
			\\
			-z_4 & z_3
		\end{pmatrix} \in \aA(\mathbb S_{\theta'}^7) \tensor \End(\C^2,\C^4)
	\end{equation*}
	is an isometry satisfying $\alpha_U\bigl(s(\sigma_1)\bigr) = s(\sigma_1)(\one_\aA \tensor U)$ for all $U \in \SU(2)$ and, in consequence, an easy computation with $z_\theta := e^{\pi\imath \theta}$ gives
	\begin{align*}
		p(\sigma_1) := s(\sigma_1)s^*(\sigma_1) = \frac{1}{2} \begin{pmatrix}
									1 + x & 0 & w_1 & w_2
									\\
									0 & 1 + x & -z_\theta w_2^* & \bar{z}_\theta w_1^*
									\\
									\bar{w}_1 & -\bar{z}_\theta w_2 & 1 - x & 0
									\\
									\bar{w}_2 & z_\theta w_1 & 0 & 1 - x
		\end{pmatrix}.
	\end{align*}
	For a Dirac operator on $\alg A(\mathbb S_\theta^4)$ one has the following construction. 
	Let $D$ be the classical Dirac operator on $\mathbb{S}^4$, that is, the unbounded self-adjoint operator on the space $L^2(\mathbb{S}^4,\mathcal{S})$ of square integrable spinors on $\mathbb{S}^4$. 
	Furthermore, let $L^2(\mathbb A^2_{\smash{\theta/2}})$ be the GNS space of the quantum 2-torus $\mathbb A^2_{\smash{\theta/2}}$ with respect to its tracial state (\cf Section~\ref{sec:QT}) and consider the essentially self-adjoint operator $D \tensor \one$ on $L^2(\mathbb{S}^4,\mathcal{S}) \tensor L^2(\mathbb A^2_{\smash{\theta/2}})$ with domain $\dom(D \tensor \one) := \dom(D) \tensor_{\text{alg}} L^2(\mathbb A^2_{\smash{\theta/2}})$. 
	The latter Hilbert space comes equipped with a natural action of the classical $\mathbb{T}^2$-torus and the restriction $D_\theta$ of the self-adjoint extension of $D \tensor \one$ to the corresponding fixed point algebra $L^2(\mathbb{S}_\theta^4,\mathcal{S})$ is the Dirac operator on $\alg A(\mathbb S_\theta^4)$. 
	It is now a consequence of~\cite[Sec.~4]{LaSu05} that the commutator $[D_\theta \tensor \one_{\C^4}, \omega(\sigma_1,\one)] = [D_\theta \tensor \one_{\C^4}, p(\sigma_1)]$ does not vanish.
\end{remark}

Summarizing, we have seen that the operator $D_h$ has a number of good properties and is almost a Dirac operator for our noncommutative principal bundle $(\aA,G,\alpha)$. 
In fact, one missing feature for $D_h$ to be a Dirac operator on $\aA$ is that it does not incorporate the geometry of the ``fibre'' $G$. 
Another one is that $D_h$ does in general not need to have compact resolvent.
In order to resolve these issues, we will add a vertical term, whose construction is the concern of the upcoming section.

\subsubsection{The vertical Dirac operator}\label{sec:verlift}

In this section we associate a vertical Dirac operator with the free C\Star-dynamical system $(\aA,G,\alpha)$. 
For this purpose, we endow $\Lie(G)$ with an $\Ad$-invariant inner product, consider a finite-dimensional \Star-representation $\pi_\spin: \CL(G) \to \End(\hH_\spin)$ of the Clifford algebra $\CL(G)$ of $\Lie(G)$, and put $F_X := \pi_\spin(X)$ for each $X \in \Lie(G)$. 

For the construction of a Dirac operator, we proceed similar to~\cite{GaGr16}. 
In greater detail, we again consider the unitary representation $\mu: G \to \mathcal U(\hH)$ and recall from Section~\ref{sec:unbounded} that for any orthonormal basis $X_1, \dots, X_n$ of $\Lie(G)$ the unbounded operator 
\begin{align*}
	D_v := \sum_{k=1}^n \partial_{X_k} \mu \tensor F_{X_k}
	\label{eq:Dv}
\end{align*}
on $\hH \tensor \hH_\text{spin}$ with domain $\dom(D_v) := \hH^\infty \tensor \hH_\text{spin}$ is essentially self-adjoint. 
For simplicity of notation, we utilize the same letter for its unique self-adjoint extension. Because the unitary representation $\mu: G \to \mathcal U(\hH)$ has finite-dimensional multiplicity spaces, $D_v$ has compact resolvent by Lemma~\ref{lem:comres2}. Throughout the remainder of this article, we refer to $D_v$ as the \emph{vertical Dirac operator} associated with $(\aA,G,\alpha)$.

Next, let us look at the unbounded operator associated with the unitary representations $u_\aA : G \to \Unitary(\hH_p)$, which is the self-adjoint extension of 
\begin{align*}
	\hat{D}_v := \sum_{k=1}^n \partial_{X_k} u_\aA \tensor F_{X_k}
\end{align*}
on $\hH_p \tensor \hH_\text{spin}$ with domain $\dom(\hat{D}_v) := \hH_p^\infty \tensor \hH_\spin$.
Furthermore, for each $X \in \Lie(G)$, we have $[\partial_X u_\aA, \pi_\aA(x)] = \partial_X \pi_\aA \bigl(\alpha(x)\bigr)$ for all $x \in \aA^\infty$. 
	
\begin{corollary}[\cf~{\cite[Prop.~2.10]{GaGr13}}]\label{cor:vertop}
	The following assertions hold for the unbounded operator $\hat{D}_v$ on  $\hH_p \tensor \hH_\text{spin}$:
	\begin{enumerate}
		\item
			$\hat{D}_v$ is self-adjoint and $\hat{D}_v = \one_{\hH_\aB} \tensor D_v$ on $\hH_p^\infty \tensor \hH_\spin$.
		\item
			$[\hat{D}_v, \pi_\aA(x)]$ is bounded for all $x \in \aA^\infty$. In greater detail, for each $x \in \aA^\infty$ we have $\pi_\aA(x) \bigl( \dom(\hat{D}_v) \bigr) \subseteq \dom(\hat{D}_v)$ and $[\hat{D}_v, \pi_\aA(x)] = \sum_{k=1}^n \partial_{X_k} \pi_\aA \bigl(\alpha(x)\bigr) \tensor F_{X_k}$.
	\end{enumerate}
\end{corollary}

We note that $\hat{D}_v = 0$ on $\hH_\aB \tensor \hH_\spin \subseteq \hH_p^\infty \tensor \hH_\spin$ and that $[\hat{D}_v, \pi_\aA(b)] = 0$ for all $b \in \aB$, which may justify the expression ``vertical''.

\pagebreak[3]
\subsubsection{The assembled Dirac operator}\label{sec:asslift}

We are finally in a position to present a spectral triple on $\aA$ that lifts the spectral triple $\DD_\aB$ and incorporates the geometry of $G$. 
Let $D_h$ be the horizontal lift of $D_\aB$ to $\hH_p$ from  Section~\ref{sec:horlift} and let $D_v$ be the vertical Dirac operator on $\hH \tensor \hH_\spin$ as introduced in the preceding section. 
Additionally, let $\gamma_\spin \in \End(\hH_\spin)$ be a self-adjoint operator satisfying $\gamma_\spin^2 = \one_{\hH_\spin}$ and $D_v (\one_\hH \tensor \gamma_\spin) = - (\one_\hH \tensor \gamma_\spin) D_v$. 
Such an operator can always be found by increasing $\hH_\spin$ if necessary. 
Then we may look at the symmetric operator
\begin{gather*}
	D_{\text{ref}} := D_\aB \tensor \one_{\hH} \tensor \gamma_\spin + \one_{\hH_\aB} \tensor D_v 
	\\
	\shortintertext{with domain}
	\dom(D_{\text{ref}}) := \dom(D_\aB) \tensor_{\text{alg}}  \dom(D_v) \subseteq \hH_\aB \tensor \hH \tensor \hH_\spin.
\end{gather*}

According to \cite[Sec.~4]{DaDo11}, its closure $\bar{D}_{\text{ref}}$ is self-adjoint with pure point spectrum consisting of countably many real eigenvalues, each with finite multiplicity, and the only limit point of their absolute values is given by $+\infty$. In particular, $\bar{D}_{\text{ref}}$ has compact resolvent.  Applying Lemma~\ref{lem:resD} and combining Theorem~\ref{thm:D_hcomm} with Corollary~\ref{cor:vertop}, we get the main results of this paper:

\begin{theorem}\label{thm:main}
	For the operator 
	\begin{equation}
		D_\aA := p \bar{D}_{\text{ref}} \, p = D_h \tensor \gamma_\spin + p(\one_{\hH_\aB} \tensor D_v)p = D_h \tensor \gamma_\spin + \hat{D}_v\label{eq:hatD}
	\end{equation}
	with domain $p \dom(\bar{D}_{\text{ref}})$ on $\hH_\aA := p(\hH_\aB \tensor \hH \tensor \hH_\spin) $ the following assertions hold:
	\begin{enumerate}
		\item
			$D_\aA$ is self-adjoint and has compact resolvent. 
		\item
			Under the hypotheses of Theorem~\ref{thm:D_hcomm}, the commutator $[D_\aA,\pi_\aA(x)]$ is bounded for all $x \in \aA_0$, where $\pi_\aA$ is understood to be amplified in the obvious way.
	\end{enumerate}
\end{theorem}

We occasionally write $\hat{D}_h := D_h \tensor \gamma_\spin$ for the horizontal part of $D_\aA$ if no confusion regarding the operator $\gamma_\spin$ can arise. Equation~\eqref{eq:hatD} then reads as $D_\aA = \hat{D}_h + \hat{D}_v$.
	
\begin{remark}\label{rem:main2}
	 Suppose the spectral triple $\DD_\aB$ is even, which amounts to saying that there is a self-adjoint unitary operator $\gamma_\aB$ on $\hH_\aB$ such that $\gamma_\aB \pi_\aB(b) = \pi_\aB(b) \gamma_\aB$ for all $b \in \aB$ and $\gamma_\aB D_\aB = - D_\aB \gamma_\aB$. Then one may work with the operator 
	\begin{align*}
		D'_{\text{ref}} := D_\aB \tensor \one_{\hH \tensor \hH_\spin} + \gamma_\aB \tensor D_v
	\end{align*}
	with domain $\dom( D'_{\text{ref}} ) := \dom(D_\aB) \tensor_{\text{alg}}  \dom(D_v) \subseteq \hH_\aB \tensor \hH \tensor \hH_\spin$. All arguments of this article considering $D_{\text{ref}}$ likewise apply to $D'_{\text{ref}}$.
\end{remark}

\begin{remark}
 	Theorem~\ref{thm:main} has a clear meaning in unbounded $KK$-theory as was kindly pointed out to us by van Suijlekom. 
 	Indeed, our Dirac operator $D_\aA$ in Equation~\eqref{eq:hatD} is an example of the Kasparov product~\cite{KaLesch13,Mes16} of the unbounded Kasparov $\aA$-$\aB$-module $(\aA_0,\pi_\aA,\hH_p,D_v)$ with the spectral triple $\DD_\aB$ and the $D_\aB$-connection induced by $D_h$. 
\end{remark}

\begin{corollary}\label{cor:main}
	$\DD_\aA := (\aA_0,\pi_\aA,\hH_\aA,D_\aA)$ is a spectral triple on $\aA$ that lifts $\DD_\aB$ in the sense of Definition~\ref{def:lift} \wrt the isometry $t: \hH_\aB \to \hH_\aA$, $t(\xi) := \xi \tensor \one_{\hH_\spin} \tensor \one_{\C^2}$.
\end{corollary}

\begin{remark}
	At this point we briefly resume the discussion in Remark~\ref{rmk:strong_lift}. Indeed, a moment's thought shows that $\Omega^1(\aB_0)$ embeds into $\Omega^1(\aA_0)$ if, for instance, $\gamma_\sigma(b) = b \tensor \one$ for all $\sigma \in \Irrep(G)$ and $b \in \aB_0$ (\cf Equation~\eqref{eq:multiplication_with_vectors}). This is, clearly, guaranteed for classical principal bundles, but also in some noncommutative situations (see, \eg,~\cite[Sec.~5]{SchWa15}).
\end{remark}

In the following sections we investigate how the lift constructed here compares to established examples.

\section{Example: Crossed products}\label{sec:crossprod}

Let $\aB$ be a unital C\Star-algebra and let $\alpha \in \Aut(\aB)$. 
This example focuses on the crossed product $\aA := \aB \rtimes_\alpha \Z$ equipped with the natural dual circle action $\hat \alpha$. 
We consider $\aB$ as a subalgebra of $\aA$ and write $v \in \aB \rtimes_\alpha \Z$ for the generator of the $\Z$-action defined by~$\alpha$. 
Then the action $\hat \alpha$ is given by 
\begin{align*}
	\hat{\alpha}_z(b) &= b \qquad \forall b \in \aB
	&
	\text{and}&
	&
	\hat{\alpha}_z(v) &= z \cdot v
\end{align*}
for all $z \in \mathbb T$. 
The fixed point algebra $\aA^{\mathbb T}$ is equal to $\aB$. 
More generally, each isotypic component $\aA(k)$, $k \in \mathbb{Z} \cong \Irrep(\mathbb{T})$, is given by $\aB v^k$. 
From this it may be concluded that the C\Star-dynamical system $(\aA,\mathbb{T},\hat{\alpha})$ is cleft and, therefore, free (\cf~\cite[Lem.~4.5]{SchWa16}). 
Moreover, we may put $\hH_k := \C$ for every $k \in \Z$ and let the isometries $u(k) := v^{-k}$, $k \in \Z$, serve as initial data for the construction provided in Section~\ref{sec:permanence}. 

Next, let $\DD_\aB := (\aB_0,\pi_\aB,\hH_\aB,D_\aB)$ be a spectral triple on $\aB$. 
For the sake of simplicity, we identify $\aB$ with $\pi_\aB(\aB) \subseteq \End(\hH_\aB)$. 
Then the dense subalgebra $\aA_0 \subseteq \aA$ from Section~\ref{sec:liftA} consists of all elements of the form $\sum_{k\in \Z} b_k v^k$ with only finitely many nonzero elements $b_k \in \aB_0$. 
The lifted \Star-representation $\pi_\aA: \aA \to \End(\hH_p)$ introduced in Section~\ref{sec:liftrep} is given by the Hilbert space $\hH_p := \ell^2(\Z, \hH_\aB)$  and the relations
\begin{align*}
	\bigl( \pi_\aA(b) \eta \bigr)_k 
	&= \alpha^{-k}(b) \eta_k 
	&
	\text{and}&
	&
	\bigl( \pi_\aA(v) \eta \bigr)_k &= \eta_{k-1}
\end{align*}
for all $b \in \aB$, $\eta \in \ell^2(\Z, \hH_\aB)$, and $k \in \Z$. 
The horizontal Dirac operator constructed in Section~\ref{sec:horlift} is given by the self-adjoint extension of the operator $D_h = D_\aB \tensor \one_{\ell^2(\Z)}$ on the domain $\dom(D_h) = \dom(D_\aB) \tensor \ell^2(\Z)$. 
The vertical Dirac operator $D_v$, as presented in Section~\ref{sec:verlift}, is the tensor product of the standard operator on the circle $\mathbb{T} \cong \R / \Z$, that is, $-\imath \frac{d}{dt}$, with the Pauli matrix 
$\sigma_2 := \begin{psmallmatrix} 0 & -\imath\\ \imath & 0 \end{psmallmatrix}$ or, equivalently,
\begin{align*}
	\bigl( D_v(\zeta) \bigr)_k = k \cdot \sigma_2 (\zeta_k)
\end{align*}
for all $\zeta = (\zeta_k)_{k \in \Z} \in \ell^2(\Z) \tensor \C^2 = \ell^2(\Z, \C^2)$.

Now, we additionally assume that $\alpha$ generates an equicontinuous group (see~\cite[Def.~3]{BeMaRe10}). In particular, we may choose
\begin{align*}
	\aB_0 := \Bigl\{ b \in \aB : \sup_{k \in \Z} \norm[\big]{[ D, \pi_\aA\bigl( \alpha^k(b) \bigr) ]}  < \infty \Bigr\}
\end{align*}
as a dense unital \Star-subalgebra of $\aB$, which ensures that the operators in Equation~\eqref{eq:commD_sigma} are uniformly bounded. Corollary~\ref{cor:main} therefore implies that $\DD_\aA:=(\aA_0,\pi_\aA,\hH_\aA,D_\aA)$ is a spectral triple on $\aA$ that lifts $\DD_\aB$ in the sense of Definition~\ref{def:lift} and builds in the geometry of circle, where $\hH_\aA := \hH_p \tensor \C^2 = \ell^2(\Z, \hH_\aB) \tensor \C^2$ and 
\begin{align*}
	\bigr( D_\aA(\eta) \bigl)_k = \bigl( D_\aB \tensor \sigma_1 + ( k \cdot \one_{\hH_\aB} )\tensor \sigma_2 \bigr)( \eta_k ) 
\end{align*}
for all $\eta = (\eta_k)_{k \in \Z} \in \hH_\aA$ and the Pauli matrix $\sigma_1 := \begin{psmallmatrix} 0 & 1\\ 1 & 0 \end{psmallmatrix}$.
In summary, our construction extends the construction provided in the seminal work~\cite[Sec.~3.4]{BeMaRe10} by Bellissard, Marcolli, and Reihani.

\section{Example: Quantum 4-tori}\label{sec:QT}

Let $\theta$ be a real skew-symmetric $4\times 4$-matrix and, for $1 \le k, \ell \le 4$, put $\lambda_{k,\ell} := \exp(2\pi \imath \theta_{k,\ell})$ for short.
In this example we consider the quantum 4-torus $\mathbb A^4_\theta$, which is the universal C\Star-algebra with unitary generators $u_1, \dots, u_4$ satisfying the relation $u_k u_\ell = \lambda_{k,\ell} u_\ell u_k$ for all $1 \le k,\ell \le 4$. 
The classical torus $\mathbb T^4$ acts naturally on $\mathbb A^4_\theta$ via the \Star-automorphisms given by $\tau(u_k) = z_k \cdot u_k$ for all $z = (z_1, \dots, z_4) \in \mathbb T^4$ and $1 \le k \le 4$. 
This is the so-called \emph{gauge action}. 
We write $L^2(\mathbb A^4_\theta)$ for the GNS space with respect to the unique $\tau$-invariant tracial state on $\mathbb A_\theta^4$ and assume $\mathbb A^4_\theta \subseteq \End\bigl( L^2(\mathbb A^4_\theta) \bigr)$. 
For this example it is expedient to consider the canonical inner product on $\Lie(\mathbb T^4) = \R^4$ and to work with the standard orthonormal basis. 
Because the gauge action is implemented by unitaries on $L^2(\mathbb A^4_\theta)$, the directional derivatives for each coordinate provide unbounded skew-symmetric operators $\partial_1, \dots, \partial_4$ on $L^2(\mathbb A^4_\theta)$. 
For a Dirac operator on $\mathbb A^4_\theta$, we take into account the irreducible \Star-representation of $\CL(\mathbb T^4)$ on the spinors $\hH_\spin := \C^2 \tensor \C^2$. 
That is, we choose operators $\sigma_1, \sigma_2, \sigma_3 \in \End(\C^2)$ satisfying $\sigma_k^2 = -1$ and $\sigma_k \sigma_\ell = - \sigma_\ell \sigma_k$ for all $1 \le k \neq \ell \le 3$ and put
\begin{align*}
	F_1 &:= \sigma_1 \tensor \sigma_3,
		&
	F_2 &:= \sigma_2 \tensor \sigma_3,
		&
	F_3 &:= \one \tensor \sigma_1,
		&
	F_4 &:= \one \tensor \sigma_2.
\end{align*}
Then the canonical Dirac operator on $\mathbb A^4_\theta$ is the self-adjoint extension of 
\begin{equation*}
	D_4 := \partial_1 \tensor F_1 + \partial_2 \tensor F_2 + \partial_3 \tensor F_3 + \partial_4 \tensor F_4
\end{equation*}
defined on some suitable domain in $L^2(\mathbb A^4_\theta) \tensor \hH_\spin$. 

Our study revolves around the restricted gauge action $\alpha:\mathbb T^2 \to \Aut(\mathbb A^4_\theta)$ defined by 
\begin{align*}
	\alpha_z(u_1) &:= u_1, 
	&
	\alpha_z(u_2) &:= u_2, 
	&
	\alpha_z(u_3) &:=z_1 \cdot u_3,
	&
	\alpha_z(u_4) &:=z_2 \cdot u_4
\end{align*}
for all $z = (z_1, z_2) \in \mathbb T^2$. 
Its fixed point algebra is the quantum 2-torus $\mathbb A^2_{\theta'}$ generated by the unitaries $u_1$ and $u_2$, where $\theta'$ denotes the real skew-symmetric $2\times 2$-matrix with upper right off-diagonal entry $\theta_{12}$. 
More generally, for each $(k, \ell) \in \Z^2$ the corresponding isotypic component is $\mathbb A^4_\theta(k, \ell)$ takes the form $u(k, \ell) \mathbb A^2_{\theta'}$ for the unitary $u(k, \ell) := u_4^\ell u_3^k$. 
In particular, the C\Star-dynamical system $(\mathbb A^4_\theta, \mathbb T^2, \alpha)$ is cleft and therefore free. Next, let $L^2(\mathbb A_{\theta'}^2)$ be the GNS space of $\mathbb A_{\theta'}^2$ with respect to its gauge-invariant trace. 
By Section~\ref{sec:representation}, the unitaries give rise to the \Star-representation $\pi_u:\mathbb A^4_\theta \to \End\bigl( L^2(\mathbb T^2) \tensor L^2(\mathbb A^2_{\theta'}) \bigr)$ uniquely determined by
\begin{align*}
	\pi_u(u_1)
	&= r_{\theta_{31}, \theta_{41}} \tensor u_1,
	&
	\pi_u(u_2)
	&= r_{\theta_{32}, \theta_{42}} \tensor u_2,
	\\
	\pi_u(u_3)
	&= z_1 \tensor \one,
	&
	\pi_u(u_3)
	&= r_{\theta_{34},0} z_2 \tensor \one
\end{align*}
where $r_t$, $t = (t_1,t_2) \in \R^2$, denotes the rotation operator given by $(r_tf)(z) := f(z_1^{t_1}, z_2^{t_2})$ for all $f \in L^2(\mathbb T^2)$ and $z = (z_1,z_2) \in \mathbb T^2$ and $z_{1/2}$ stands for the multiplication operator $(z_{1/2}) f(z) := z_{1/2} f(z)$ for all $f \in L^2(\mathbb T^2)$  and $z \in \mathbb T^2$.  
Now, a standard computation establishes that this \Star-representation is unitarily equivalent to the GNS-representation. 

In order to apply the construction of Section~\ref{sec:liftdirac}, we pick the canonical Dirac operator $D_2 = \partial_1 \tensor F_1 + \partial_2 \tensor F_2$ on $L^2(\mathbb A^2_{\theta'}) \tensor \C^2$. Then Section~\ref{sec:horlift} immediately provides the horizontal lift $D_h$ of $D_2$, which is the self-adjoint extension of 
\begin{equation*}
	D_h = \one \tensor \partial_1 \tensor \sigma_1 + \one \tensor \partial_2 \tensor \sigma_2
\end{equation*}
on some suitable dense domain in $L^2(\mathbb T^2) \tensor L^2(\mathbb A^2_{\theta'}) \tensor \C^2$. 
The vertical Dirac operator constructed in Section~\ref{sec:verlift} is given by
\begin{equation*}
	D_v = \partial_1^{\mathbb T^2} \tensor \sigma_1 + \partial_2^{\mathbb T^2} \tensor \sigma_2
\end{equation*}
on some suitable dense domain in $L^2(\mathbb T^2) \tensor \C^2$, where $\partial_1^{\mathbb T^2}$ and $\partial_2^{\mathbb T^2}$ denote the derivatives of the translation on $L^2(\mathbb T^2)$ along the respective coordinates.  
According to Section~\ref{sec:asslift}, the assembled Dirac operator on $\mathbb A^4_\theta$ then reads 
\begin{align*}
	D_{\mathbb A^4_\theta} 
	&= D_h \tensor \sigma_3 + \one \tensor D_v
	\\
	&= \one \tensor \partial_1 \tensor F_1
	+ \one \tensor \partial_2 \tensor F_2
	+ \partial_1^{\mathbb T^2} \tensor \one \tensor F_3
	+ \partial_2^{\mathbb T^2} \tensor \one \tensor F_4
\end{align*}
on some suitable dense domain in $L^2(\mathbb T^2) \tensor L^2(\mathbb A^2_{\theta'}) \tensor \hH_\spin$. A straightforward verification finally shows that this operator coincides with the canonical Dirac operator of $\mathbb A^4_\theta$ up to the unitary equivalence to the GNS representation.

\begin{remark}
	Let $\gamma_{k, \ell}$ for $(k, \ell) \in \Z^2 = \Irrep(\mathbb T^2)$ be the coaction of the associated factor system. 
	It is straightforward to check that $\gamma_{k, \ell}$ is, in fact, the gauge action $\gamma_{k,\ell}= \tau_z$ for $z = (\lambda_{31}^k \lambda_{41}^\ell, \lambda_{32}^k \lambda_{42}^\ell)$.
	The gauge-invariance of the Dirac operator $D_2$ therefore yields
	\begin{equation}\label{eq:qTorus_gamma}
		[D_2, \gamma_{k, \ell}(x) \tensor \one_{\C^2}] = \gamma_{k, \ell} \tensor \id \bigl( [D_2, x \tensor \one_{\C^2}] \bigr)
	\end{equation}
	for all $x \in \mathbb A^2_{\theta'}$ and $k, \ell \in \Z$. 
	This and the fact that the cocycle of the factor system takes values in $\C \cdot \smash{\one_{\mathbb A^2_{\theta'}}}$ entail that the hypotheses of Theorem~\ref{thm:D_hcomm} are fulfilled. 
	Furthermore, coming back to Remark~\ref{rmk:strong_lift}, it is an easy task to show that Equation~\eqref{eq:qTorus_gamma} implies that $\Omega^1(\mathbb A^2_{\theta'})$ naturally embeds into $\Omega^1(\mathbb A^4_{\theta})$.
\end{remark}

\section{Example: Homogeneous spaces}\label{sec:Homogeneous}

In this example we consider a compact Lie group $G$ together with a closed subgroup $H \leq G$ acting on $G$ from the right. This gives rise to a principal $H$-bundle over $G / H$. Algebraically we look at the C\Star-algebra $\Cont(G)$ endowed with the action $r_h$, $h \in H$, given by $(r_h f)(g) := f(g h)$ for all $g \in G$ and $h \in H$, which is certainly free in the sense of Ellwood according to~\cite[Prop.~7.1.12 and Thm.~7.2.6]{Phi87}). We identify the corresponding fixed point algebra with $\Cont(G/H)$.

In what follows, we write $\pi_G:\Cont(G) \to \End\bigl( L^2(G) \bigr)$ and $\pi_H: \Cont(H) \to \End \bigl( L^2(H) \bigr)$ for the \Star-representation by multiplication operators, respectively. 
For the adjoint action of $G$ on $\Lie(G)$ we write $\Ad_{g}$, $g \in G$, and we use the same notation for its extension to the Clifford algebra $\CL(G)$ by algebra automorphisms.

\subsection{Regarding freeness}

To establish freeness algebraically, we regard $L^2(H) \subseteq  L^2(G)$ as the subspace of functions with support in $H$ and
identify the multiplicator algebra $\Mul \bigl( \Cont(G) \tensor \Com \bigl(L^2(H), L^2(G) \bigr) \bigr)$ with the algebra of strongly continuous bounded functions $f: G \to \End\bigl( L^2(H), L^2(G) \bigr)$, which we denote by $\Cont_{\text{sb}} \bigl( G, \End\bigl( L^2(H), L^2(G) \bigr) \bigr)$ (\cf~\cite[Lem.~2.57]{Rae98}). 
Then the element $s \in \Cont_{\text{sb}} \bigl( G, \End( L^2(H), L^2(G)) \bigr)$ defined by 
\begin{equation*}
	s(g) := r_g|_{L^2(H)}
\end{equation*}
satisfies the conditions of Lemma~\ref{lem:coisometry} for the Hilbert space $\hH := L^2(G)$ equipped with the left translations $\mu_h:= \lambda_h |_{L^2(H)}$, $h \in H$.
We set $p:= s s^*$ for short.
The \Star-representation $\pi_s: \Cont(G) \to \Cont_{\text{sb}} \bigl( G /H, \End(\hH) \bigr)$ associated with $s$, introduced in Lemma~\ref{lem:pi_S}, reads as
\begin{equation*}
	\pi_s(f)(g) := s(g) \; \pi_G \bigl( jf(g, \; \cdot ) \bigr)  \; s(g)^*,
	\qquad
	g \in G,
\end{equation*}
where $j:\Cont(G) \to \Cont(G \times H)$ is given by $(jf)(g, h) := f(g h^{-1})$ for all $g \in G$ and $h \in H$.

\subsection{Dirac operators}

For the Dirac operators on $G$ and $G/H$ we follow~Rieffel \cite{Rieffel08} up to a conventional sign. 
For the convenience of the reader we briefly recall their construction. 
For this purpose, we endow $\Lie(G)$ with an $\Ad$-invariant inner product and decompose it into the direct sum  of $\Lie(H)$ and its orthogonal complement, denoted by $\Lie(G/H)$. 
We utilize $P_H$ and $P_{G/H}$ to be the corresponding orthogonal projections onto $\Lie(H)$ and $\Lie(G/H)$, respectively, and we write $\CL(G/H)$ for the Clifford algebras of $\Lie(G/H)$.

\subsubsection{Dirac operators on \texorpdfstring{$G$}{the large group}}\label{sec:homogeneous_Dirac_big}

For the rest of the paper, we let $F_X: \CL(G) \to \CL(G)$ stand for the multiplication by a vector $X \in \Lie(G)$, \ie, $F_X(\varphi) := X \cdot \varphi$ for all $\varphi \in \CL(G)$. To construct a Dirac operator on $G$, we first look at the representation of~$\Cont(G)$ on the Hilbert space 
\begin{equation*}
	\hilb S^2(G) := L^2\bigl( G, \CL(G) \bigr)
\end{equation*}
by pointwise multiplication operators and identify $\alg T(G) := \Cont^\infty \bigl( G, \Lie(G) \bigr)$ with the space of smooth sections of the tangent bundle of $G$ in terms of left translations. Obviously, $\Cont^\infty(G)$ acts on $\alg T(G)$ by pointwise multiplication. Second, we define both a connection and a Clifford multiplication on the subspace $\hilb S^\infty(G) \subseteq \hilb S^2(G)$ of all smooth functions by putting for each $X \in \alg T(G)$ and $\varphi \in \hilb S^\infty(G)$
\begin{align*}
	\partial_X^{G} \varphi (g) 
	&:= \dt \varphi \bigl( \exp (-t X(g) ) g \bigr)
	&
	&\text{and} 
	&
	F_X \varphi(g) 
	:= F_{X(g)} \varphi(g)
\end{align*}
for all $g \in G$, respectively. Finally, we fix a standard module frame $(X_k)_k$ of $\alg T(G)$ and extend the essentially self-adjoint operator
\begin{equation*}
	D_G := \sum_k F_{X_k} \; \partial_{X_k}^G
\end{equation*}
from the domain $\hilb S^\infty(G)$ to a self-adjoint Dirac operator, which we again denote by $D_G$. 
This operator is independent of the choice of standard module frame (see~\cite[Sec.~8]{Rieffel08}). 
However, there is a convenient choice of a standard module frame of $\alg T(G)$ in order to work with the quotient. 
Since the identification of $\alg T(G)$ with the tangent bundle is done via left translation and $H$ acts on $G$ from the right, it is expedient to decompose $\Lie(G)$ in the range of elements of $\alg T(G)$ at a point $g \in G$ into 
\begin{equation*}
	\Lie(G) = \Ad_g \bigl( \Lie(G/H) \bigr) \oplus \Ad_g \bigl( \Lie(H) \bigr).
\end{equation*}
We may then fix an orthonormal basis $(X_k)_k$ of $\Lie(G)$ and put
\begin{align} \label{eq:module_frame}
	Y_k(g) &:= \Ad_g P_{G/H} \Ad_g^{-1}(X_k)
		   &
		   &\text{and} 
		   &
	Z_k(g) &:= \Ad_g P_H \Ad_g^{-1}(X_k)
\end{align}
for all $g \in G$. Together $(Y_k)_k$ and $(Z_k)_k$ form a standard module frame of $\alg T(G)$ and we call the corresponding summands of $D_G$, 
\begin{align*}
	D_h &:= \sum_k F_{Y_k} \; \partial_{Y_k}^G
		& 
		&\text{and}
		& 
	D_v &:= \sum_k F_{Z_k} \; \partial_{Z_k}^G
\end{align*}
the \emph{horizontal part} and the \emph{vertical part} of $D_G$, respectively.

\subsubsection{Dirac operators on \texorpdfstring{$G/H$}{the quotient}} 

For a Dirac operator on $G/H$ we consider the Hilbert space 
\begin{equation*}
	\hilb S^2(G/H) := \{ \varphi \in L^2\bigl( G, \CL(G/H) \bigr) : \varphi(g h) = \Ad_h^{-1} \bigl( \varphi(g)\bigr) \quad \forall g \in G, h \in H \}
\end{equation*}
and let $\pi_{G/H}: \Cont(G/H) \to \End\bigl( \hilb S^2(G/H) \bigr)$ stand for the \Star-representation of $\Cont(G/H)$ on $\hilb S^2(G/H)$ by multiplication operators. Moreover, we denote by $\hilb S^\infty(G/H) \subseteq \hilb S^2(G/H)$ the subspace of smooth functions, point out that the module
\begin{equation*}
	\alg T(G/H) := \{ X \in \Cont^\infty \bigl( G, \Lie(G/H) \bigr) : X(g h) = \Ad_h^{-1} \bigl( X(g) \bigr) \quad \forall g \in G, h \in H \}
\end{equation*} 
can be recognized as smooth sections of the tangent bundle of~$G/H$, and write $F_X$ for the pointwise Clifford multiplication on $\hilb S^\infty(G/H)$ by some $X \in \alg T(G/H)$. We also bring to mind that $\hilb S^2(G/H)$ and $\alg T(G/H)$ can naturally be regarded as subspaces of $\hilb S^2(G)$ and $\alg T (G)$, respectively, with respect to the embeddings $\varphi \mapsto \hat \varphi$ and $X \mapsto \hat X$ given by
\begin{align*}
	\hat\varphi(g) 
	&:= \Ad_g \bigl( \varphi(g) \bigr)
	&
	&\text{and}
	&
	\hat X(g)
	&:= \Ad_g \bigl( X(g) \bigr)
\end{align*}
for all $g\in G$, respectively. Correspondingly, each $X \in \alg T(G/H) \subseteq \alg T (G)$ gives rise to a connection $\partial_X^{G/H} : \hilb S^\infty(G/H) \to \hilb S^\infty(G/H)$ by putting
\begin{equation*}
	\partial_X^{G/H} \varphi(g) := \dt \varphi \bigl( g \exp (-t X(g)) \bigr)
\end{equation*}
for all $g \in G$. We may now choose any standard module frame $(Y_k)_k$ for $\alg T(G/H)$ and define an operator $D_{G/H} : \hilb S^\infty(G/H) \to \hilb S^\infty(G/H)$ via
\begin{equation*}
	D_{G/H} := \sum_k F_{Y_k}^{} \partial_{Y_k}^{G/H}.
\end{equation*}
It follows from~\cite[Cor.~8.5]{Rieffel08} that $D_{G/H}$, the so-called Hodge-Dirac operator, is formally self-adjoint, and hence it admits a self-adjoint, possibly unbounded extension, for which we use the same letter $D_{G/H}$.

\subsubsection{The horizontal and the vertical part of the lifted Dirac operator} 

The task is now to lift the Hodge-Dirac operator $\smash{D_{G/H}}$ to a Dirac type operator on $G$. 
For this purpose, we consider the possibly degenerated \Star-representation of $\Cont(G)$ on the Hilbert space $\hilb S^2(G/H) \tensor \hH \tensor \CL(H)$ given by 
\begin{equation*}
	\pi(f) := \bigl( \pi_{G/H} \tensor \id_{\End(\hH)} \bigr) \bigl( \pi_s(f) \bigr) \tensor \one_{\CL(H)}, \qquad f \in \Cont(G).
\end{equation*}
To simplify notation, we swap tensor factors and restrict $\pi$ to its non-degenerate range, which is the subspace $\hilb S^2 \subseteq L^2\bigl( G, \hH \tensor \CL(G/H) \tensor \CL(H) \bigr)$ of functions $\varphi$ with
\begin{equation}
	\label{eq:equivariant}
	\varphi(g h) = \bigl( p(g) \tensor \Ad_h^{-1} \tensor \Ad_h^{-1} \bigr) \varphi(g) \qquad \forall g \in G, h \in H.
\end{equation}
Moreover, we consider the subspace $\hilb S^\infty \subseteq \hilb S^2$ of functions with range in the algebraic tensor product $\hH^\infty \tensor \CL(G/H) \tensor \CL(H)$, where $\hH^\infty$ denotes the smooth domain of the unitary representation $\mu: H \to \Unitary(\hH)$. 
Following the construction of Section~\ref{sec:permanence}, we infer that the horizontal part $\hat D_h : \hilb S^\infty \to \hilb S^\infty$ of the Dirac operator is given by
\begin{align*}
	\hat D_h \varphi &= (p \tensor \one_{\CL(G/H)} \tensor \one_{\CL(H)})(1_\hH \tensor D_{G/H} \tensor \Omega) \varphi,
\end{align*}
$\Omega$ being the grading operator on $\CL(H)$.
The vertical part $\hat{D}_v : \hilb S^\infty \to \hilb S^\infty$ is given by
\begin{equation*}
	\hat D_v \varphi(g) 
	= \sum_k \dt ( \mu_{\exp(t X_k)} \tensor \one_{\CL(G/H)} \tensor F_{X_k} ) \varphi(g), \qquad g \in G,
\end{equation*}
for an arbitrarily chosen orthonormal basis $(X_k)_k$ of $\Lie(H)$. The lifted Dirac operator $\hat D: \hilb S^\infty \to \hilb S^\infty$ on $G$ established in Theorem~\ref{thm:main} thus takes the form $\hat{D}:= \hat D_h + \hat D_v$.

\subsubsection{Comparing the Dirac operators on \texorpdfstring{$G$}{G}}

We are finally in a position to compare the lifted Dirac operator $\hat D$ with the canonical Dirac operator $D_G$. 
To this end, we let $W: \CL(G/H) \tensor \CL(H) \to \CL(G)$ stand for the unitary map defined by $W(X \tensor Y) := X \cdot \Omega(Y)$ for all $X \in \CL(G/H)$ and $Y \in \CL(H)$. 
We also write $\ev_1:L^2(H) \supseteq \Cont^\infty(H) \to \C$ for the evaluation at the unit element of $H$, that is, $\varphi \mapsto \varphi(1)$. 
Then the map $U:\hilb S^\infty \to L^2 \bigl( G, \CL(G) \bigr)$ given by
\begin{align}\label{eq:homogoneous_unitary}
	(U \varphi)(g) 
	&:= U(g) \varphi(g)
	&
	&\text{with}
	&
	U(g) 
	&:= \ev_1 s(g)^* \tensor \Ad_g W
\end{align}
for all $\varphi \in \hilb S^\infty$ and $g \in G$ extends to a unitary map $U:\hilb S^2 \to \hilb S^2(G)$.

\begin{lemma}
	The covariant representations $(\pi_G, r_h)$ and $(\pi, \mu)$ are unitarily equivalent. 
	More precisely, for all $f \in \Cont(G)$ and $h \in H$ we have 
	\begin{align*}
		U \pi(f) 
		&= \pi_G(f) U
		&
		&\text{and}
		&
		r_h U 	 
		&= U \mu_h.
	\end{align*}
\end{lemma}
\begin{proof}
	Let us first fix $f \in \Cont(G)$. 
	Then for each $g \in G$ and $\varphi \in \Cont^\infty(H)$ we find
	\begin{equation*}
		(\ev_1 \circ \pi_H) \bigl( jf(g, \;\cdot \;) \bigr) \varphi 
		= f(g) \cdot \varphi(1) 
		= f(g) \cdot \ev_1(\varphi).
	\end{equation*}
	From this, for each $\varphi \in \hilb S^\infty(G)$ and $g \in G$ we obtain that
	\begin{align*}
	 	\bigl( U \pi(f) \varphi \bigr)(g)
		&= \bigl( \ev_1 s(g)^* \tensor \Ad_g W \bigr) \; 
		\\
		& \qquad 
		\Bigl(  \bigl( s(g) \, \pi_H \bigl( jf(g, \; \cdot \;) \bigr) \, s(g)^* \bigr) \tensor \one_{\CL(G/H)} \tensor \one_{\CL(H)} \Bigr) \varphi(g)
		\\
		&= f(g) \cdot  \bigl( \ev_1 s(g)^* \tensor \Ad_g W \bigr) \varphi(g)
		= \bigl( \pi_G(f) U \varphi \bigr)(g).
	\end{align*}
	In other words, we have $U \pi(f) = \pi_G(f) U$ as claimed. 
	To deal with the second assertion, we fix $h \in H$ and note that
	\begin{equation}\label{eq:nu_to_mu}
		\ev_1 s(g h)^* 
		\overset{\eqref{eq:SOPequivariance}}= \ev_1 r_h^* s(g)^*
		\ev_1 \lambda_h s(g)^* 
		\overset{\eqref{eq:SOPcommuting}}= \ev_1 s(g)^* \mu_h
	\end{equation}
	for all $g \in G$. Hence for each $\varphi \in \hilb S^\infty$ and $g \in G$ we deduce that
	\begin{align*}
		(r_h U \varphi)(g)
		&\overset{\hphantom{\eqref{eq:nu_to_mu}}}=
		\bigl( \ev_1 s(g h)^*  \tensor \Ad_{g h} W \bigr) \varphi(g h)
		\\
		&\overset{\eqref{eq:nu_to_mu}}=
		\bigl( \ev_1 s(g)^* \tensor \Ad_g W \bigr) (\mu_h \tensor \Ad_h \tensor \Ad_h) \varphi(g h)
		\\
		&\overset{\eqref{eq:equivariant}}= 
		\bigl( \ev_1 s(g)^* \tensor \Ad_g W\bigr)( \mu_h \tensor \one_{\CL(G/H) \tensor \CL(H)}) \varphi(g)
		= (U \mu_h \varphi)(g).
	\end{align*}
	That is, $r_h U = U \mu_h$, and so the proof is complete.
\end{proof}

\pagebreak[3]
\begin{theorem}\label{thm:compare}
	For the module frame in Equation~\eqref{eq:module_frame}, and hence for all module frames, the following assertions hold:
	\begin{enumerate}
		\item
			$D_h = U \hat D_h U^* + \sum_k  F_{Y_k} (d_{Y_k}U) U^*$ on $\hilb S^\infty(G)$. 
		\item
			$D_v = U \hat D_v U^*$ on $\hilb S^\infty$.
	\end{enumerate}
	Here, for each $Y \in \alg T(G)$, we put $(d_Y U) \varphi(g) := \dt U \bigl( \exp(-t Y(g))g \bigr) \varphi(g)$
	for all $\varphi \in \hilb S^\infty$ and $g\in G$.
\end{theorem}

\begin{remark}\label{rem:compare}
	The operator $\sum_k F_{Y_k} (d_{Y_k} U) U^*$ appearing in Lemma~\ref{thm:compare} above commutes with every element $\pi(x)$ for $x \in \Cont(G)$. Consequently, this additional term does not effect any of the commutators $[D_h^n, \pi(x)]$, $n \in \N$.
\end{remark}

\begin{proof}
	For the computation we fix an orthonormal basis of $\Lie(G)$ and utilize the standard module frame of $\alg T(G)$ introduced in Equation~\eqref{eq:module_frame}. 
	Accordingly, we choose $(Y'_k)_k$ with $Y'_k(g) := \Ad_g^{-1} \bigl( Y_k(g) \bigr) \in \Lie(G/H)$, $g \in G$, and  $(Z'_k)_k$ with $Z'_k(g) := \Ad_g^{-1} \bigl( Z_k(g) \bigr) \in \Lie(H)$, $g \in G$, as standard module frames of $\alg T(G/H)$ and $\alg T(H)$, respectively. 
	\begin{enumerate}
		\item
			Let $\varphi \in \hilb S^\infty$ and $g \in G$. 
			Applying the Leibniz rule, for each $k$ we deduce that
			\begin{equation}\label{eq:U_derivatives}
				\begin{split}
					\Bigl( \partial^G_{Y_k} U \varphi \Bigr)(g)
					&= \dt U \bigl( e^{-t Y_k(g)} g \bigr) \varphi \bigl(e^{-t Y_k(g)} g \bigr)
					\\
					&= (d_{Y_k}U \varphi)(g) + \Bigl( U \partial^{G/H}_{Y'_k} \varphi \Bigr)(g).
				\end{split}
			\end{equation}
			Furthermore, since $F_Y W = W(F_Y \tensor \Omega)$ for every $Y \in \Lie(G/H)$, for each $k$ we find
			\begin{equation}\label{eq:U_Clifford}
				\begin{split}
					F_{Y_k(g)} U(g) 
					&= \ev_1 s(g)^* \tensor F_{Y_k(g)} \Ad_g W
					= \ev_1 s(g)^* \tensor \Ad_g F_{Y'_k(g)} W
					\\
					&= \ev_1 s(g)^* \tensor \Ad_{g} W \bigl( F_{ Y'_k(g)} \tensor \Omega \bigr) 
					= U(g) \bigl( F_{Y_k'(g)} \tensor \Omega \bigr).
				\end{split}
			\end{equation}
			From this the asserted equation follows:
			\begin{align*}
				U^* D_h U \varphi
				&= \sum_k U^* F_{Y_k} \partial^G_{Y_k} U \varphi
				\overset{\eqref{eq:U_derivatives}}= 
				\sum_k U^* F_{Y_k} \Bigl( U \partial^{G/H}_{Y'_k} + d_{Y_k}U \Bigr) \varphi 
				\\
				&\overset{\eqref{eq:U_Clifford}}=
				\sum_k \bigl( F_{Y'_k} \tensor \Omega \bigr) \partial^{G/H}_{Y'_k} \varphi + \sum_k U^* F_{Y_k} d_{Y_k}U \varphi 
				\\
				&= \hat D_h \varphi + \sum_k U^* F_{Y_k} d_{Y_k} U \varphi.
			\end{align*}
	\item
		Let $\varphi \in \hilb S^\infty$ and $g \in G$. By Equation~\eqref{eq:equivariant}, for each $k$ we have
		\begin{equation}\label{eq:phi_Ad}
			\varphi \big( e^{-t Z_k(g)} g \bigr) 
			= \varphi \bigl( g e^{-t Z'_k(g)} \bigr)
			= \bigl( \one_{\hH} \tensor \Ad_{\exp(tZ'_k(g))} \tensor \Ad_{\exp( t Z'_k(g))} \bigr) \varphi(g).
		\end{equation}
		For the vertical derivatives it may thus be concluded that
		\begin{align*}
			\MoveEqLeft
			\Bigl( \partial_{Z_k}^G U \varphi \Bigr)(g)
			\overset{\hphantom{\eqref{eq:nu_to_mu}}}=
			\dt U \bigl(e^{-t Z_k(g)} g \bigr) \varphi \bigl( e^{-tZ_k(g)} g\bigr)
			\\
			&\overset{\hphantom{\eqref{eq:nu_to_mu}}}=
			 \dt U \bigl( g e^{-t Z'_k} \bigr) \varphi \bigl( g e^{-t Z'_k(g)} \bigr)
			\\
			&\overset{\eqref{eq:nu_to_mu}}=
			\dt U(g) \bigl( \mu_{\exp(-tZ'_k(g))}\tensor \Ad_{\exp(-tZ'_k(g))} \tensor \Ad_{\exp(-tZ'_k(g))} \bigr) \varphi \bigl( g e^{-t Z'_k(g)} \bigr)
			\\
			&\overset{\eqref{eq:phi_Ad}}=
			\Bigl( U \bigl( \partial_{Z'_k} \mu \tensor \one_{\CL(G/H) \tensor \CL(H)} \bigr) \varphi \Bigr)(g),
		\end{align*}
		where $\partial_X \mu$, $X \in \alg T(H)$, denotes the operator on $L^2(G, \hH^\infty) \subseteq L^2(G,\hH)$ given by 
		\begin{equation*}
			(\partial_X \mu \, \varphi)(g) := \dt \mu_{\exp(t X(g))} \varphi(g).
		\end{equation*}
		Furthermore, since $F_Z W = W(\one \tensor F_Z)$ for all $Z \in \alg T(H) \subseteq \alg T(G)$, for each $k$ we see in much the same way as above that $F_{Z_k} U \varphi = U \bigl( \one \tensor F_{Z'_k} \bigr) \varphi$. 
		Hence
		\begin{align*}
			U^* D_v U \varphi
			&= \sum_k U^* F_{Z_k} \partial_{Z_k}^G U \varphi
			= \sum_k U^* F_{Z_k} U \partial_{Z'_k} \, \mu \varphi
			= \sum_k F_{Z'_k} \partial_{Z'_k} \, \mu \varphi.
		\end{align*}
		This gives $D_v = U \hat D_v U^*$ as claimed, because the choice of the standard module frame of $\alg T(G)$ in the definition of $\hat D_v$ is irrelevant. 
		\qedhere
	\end{enumerate}
\end{proof}

\appendix

\section{Complementary results and proofs}\label{sec:proofs}

In this appendix we provide complementary results and proofs for the sake of completeness.

\begin{lemma}\label{lem:comres1}
	Let $D$ be an unbounded self-adjoint operator on a Hilbert space $\hH$ with compact resolvent and let $x \in \End(\hH)$. Then $D + x$ has compact resolvent.
\end{lemma}
\begin{proof}
	Let $\lambda \in \imath \R$ such that $\Vert x(\lambda-D)^{-1} \Vert <1$. Then the operator $\bigl( \one-x(\lambda-D)^{-1} \bigr)$ is invertible and thus the identity $\bigl( \lambda-(D+x) \bigr) = \bigl( \one-x(\lambda-D)^{-1} \bigr) (\lambda-D)$ implies that $\lambda$ is also a regular value for $D+x$. In particular, we deduce that
	\begin{equation*}
		\bigl( \lambda-(D+x) \bigr)^{-1} = (\lambda-D)^{-1} \bigl( \one-x(\lambda-D)^{-1} \bigr)^{-1}.
	\end{equation*}
	Since $(\lambda - D)^{-1}$ is compact, this operator is compact, too. Hence $D+x$ has compact resolvent.
\end{proof}

\begin{proof}[Proof of Lemma~\ref{lem:resD}]
	We first note that $D_p$ has dense domain, because $p$ is contractive. That is, $D_p$ is indeed an unbounded operator on $p(\hH)$.
	\begin{enumerate}
	\item 
		For all $\xi, \eta \in \dom(D_p)$ we have
		\begin{equation*}
			\langle \xi, D_p \eta \rangle = \langle \xi, D \eta \rangle = \langle D \xi, \eta \rangle = \langle D_p \xi, \eta \rangle,
		\end{equation*}
        which entails that $D_p$ is symmetric. Therefore, it suffices to prove that $\dom(D_p^*) \subseteq \dom(D_p)$. To do this, let us fix $\xi \in \dom(D_p^*)$. Then for each $\eta \in \dom(D)$ we find
		\begin{equation*}
			\langle \xi, D \eta \rangle 
			= \langle \xi, pD \eta \rangle 
			= \langle \xi, Dp \eta \rangle - \langle \xi, [D,p]\eta \rangle 
			= \langle \xi, D_p (p\eta) \rangle - \langle \xi, [D,p] \eta \rangle.
		\end{equation*}
		Since $[D, p]$ is bounded, the right-hand side of the above equation is a continuous function of $\eta$. Consequently, $\xi \in \dom(D^*)= \dom(D)$, and $\xi = p\xi$ thus belongs to $p \, \dom(D) = \dom(D_p)$.
	\item
		Suppose $D$ has compact resolvent. Due to $p \, \dom(D) \subseteq \dom(D)$, we can assert that the linear operator $\tilde D_p := p D p + (1-p)D(1-p)$ on $\hH$ is a well-defined unbounded operator with domain $\dom(\tilde D_p) := \dom(D)$. Moreover, rewriting $\tilde D_p$ as
		\begin{equation*}
			\tilde D_p = D - (1-p)Dp - pD(1-p) = D + \bigl[ [D,p],(1-p) \bigr]
		\end{equation*}
		we see that $\tilde D_p$ has compact resolvent by Lemma~\ref{lem:comres1}. From this and the fact that $\tilde D_p$ commutes with $p$ it follows that $D_p = p \tilde D_p$ has compact resolvent, too. \qedhere		
	\end{enumerate}
\end{proof}

\begin{proof}[Preface to Lemma~\ref{lem:comres2}]
	We begin by briefly reviewing the Dirac operator associated with the left regular representation $\lambda: G \to \alg U \bigl(L^2(G)\bigr)$, $g \mapsto \lambda_g$ and the \Star-representation $\pi_\spin$, which is the self-adjoint extension of
	\begin{equation*}
		D_G = \sum_{k=1}^n \partial_{X_k} \lambda \tensor F_{X_k}
	\end{equation*}
	defined on some suitable domain in $L^2(G) \tensor \hH_\spin$. 
	Decomposing $L^2(G)$ into its isotypic components $L^2(G) = \bigoplus_{\sigma \in \Irrep(G)} V_\sigma \tensor \bar V_\sigma$ such that the left and right translation read as
\begin{align*}
	\lambda_g &= \bigoplus_{\sigma \in \Irrep(G)} \one_{V_\sigma} \tensor \bar \sigma_g
	& 
	&\text{and}
	& 
	r_g &= \bigoplus_{\sigma \in \Irrep(G)} \sigma_g \tensor \one_{\bar V_\sigma}
\end{align*} 
for all $g \in G$, respectively, we easily infer that each eigenspace $E_\nu(D_G)$ for an eigenvalue $\nu \in \R$ of $D_G$ takes the form
	\begin{equation*}
		E_\nu(D_G) = \bigoplus_{\sigma \in \Irrep(G)} V_\sigma \tensor E_\nu(\bar \sigma),
	\end{equation*}
	where, for each $\sigma \in \Irrep(G)$, $E_\nu(\sigma)$ denotes the $\nu$-eigenspace of the self-adjoint operator $D_\sigma := \sum_{k=1}^n \partial_{X_k} \sigma \tensor F_{X_k}$ on $V_\sigma \tensor \hH_\spin$.
Since $D_G$ has compact resolvent (see, \eg,~\cite{friedrich2000,Rieffel08}), each eigenspace $E_\nu(D_G)$ is finite-dimensional, and hence for a given $\nu$ only finitely many $E_\nu(\sigma)$, $\sigma \in \Irrep(G)$, are non-zero.
\end{proof}
	
\begin{proof}[Proof of Lemma~\ref{lem:comres2}]
	Let $\hH_\sigma$, $\sigma \in \Irrep(G)$, be the finite-dimensional multiplicity spaces of the unitary representation $u: G \to \mathcal U(\hH)$ and, for convenience, let us assume that
	\begin{align*}
		\hH &= \bigoplus_{\sigma \in \Irrep(G)} \hH_\sigma \tensor \bar V_\sigma, 
		&
		&\text{and}
		&
		u_g & = \bigoplus_{\sigma \in \Irrep(G)} \one_{\hH_\sigma} \tensor \bar \sigma_g.
	\end{align*}
	It is a simple matter to check that $D$ commutes with the elements of the C\Star-subalgebra $\bigoplus_{\sigma \in \Irrep(G)} \End(\hH_\sigma) \tensor \one_{\bar{V}_\sigma} \tensor \one_{\hH_\text{spin}}
	\subseteq \End(\hH \tensor \hH_\text{spin})$, and so do the spectral projections of $D$.
	It follows that the eigenspace $E_\nu(D)$ of any eigenvalue $\nu$ takes the form
	\begin{align*}
		E_\nu(D) = \bigoplus_{\sigma \in \Irrep(G)} \hH_\sigma \tensor E_\nu(\bar{\sigma}).
	\end{align*}
	Since each $\hH_\sigma$, $\sigma \in \Irrep(G)$ is finite-dimensional and only finitely many $E_\nu(\sigma)$, $\sigma \in \Irrep(G)$, are non-zero, we can assert that $E_\nu(D)$ is finite-dimensional. 
	The same argument shows that  $\text{spec}(D) \subseteq \text{spec}(D_G)$. 
	In particular, the eigenvalues of $D$ form a discrete set. 
\end{proof}

\begin{proof}[Proof of Lemma~\ref{lem:coisometry}, ``$(b) \Rightarrow (a)$'']
Let $\mu: G \to \mathcal U(\hH)$ be a unitary representation with finite-dimensional multiplicity spaces, let's say, $\hH_\sigma$, $\sigma \in \Irrep(G)$, like in Equation~\eqref{eq:decompH}, let $(\pi,u)$ be a faithful covariant representation of $(\aA,G,\alpha)$ on some Hilbert space $\hH_\aA$, and let $s \in \End(\hH_\aA \tensor L^2(G), \hH_\aA \tensor \hH)$ be an isometry satisfying the Equations~\eqref{eq:SOPleft},~\eqref{eq:SOPequivariance}, and~\eqref{eq:SOPcommuting}. Our analysis starts with the simple observation that Equation~\eqref{eq:SOPcommuting} is equivalent to saying that $s$ is an intertwiner between the representations $\one_\aA \tensor \lambda$ and $\one_\aA \tensor \mu$ on $\hH_\aA \tensor L^2(G)$ and $\hH_\aA \tensor \hH$, respectively. 
	In particular, $s$ maps each multiplicity space of $\one_\aA \tensor \lambda$ into the corresponding multiplicity space of $\one_\aA \tensor \mu$ or, to be more precise, $\hH_\aA \tensor V_\sigma$ into $\hH_\aA \tensor \hH_\sigma$ for all $\sigma \in \Irrep(G)$. Consequently, $s$ may be disassembled into a family of isometries 
	\begin{equation*}
	s(\sigma) \in \End(\hH_\aA \tensor V_\sigma,\hH_\aA \tensor \hH_\sigma), \qquad \sigma \in \Irrep(G).
	\end{equation*}
	We proceed with a fixed $\sigma \in \Irrep(G)$ and write $p_\sigma$ for the orthogonal projection onto the isotypic component $ L^2(G)(\bar \sigma) = V_\sigma \tensor \bar V_\sigma$. Then $p_\sigma \in \Com \bigl( L^2(G) \bigr)$ and $s (1_\aA \tensor p_\sigma)$ lies in $\aA \tensor \Com(L^2(G),\hH)$, the latter being a consequence of Equation~\eqref{eq:SOPleft}. Moreover, a moment's thought shows that restricting $s (1_\aA \tensor p_\sigma)$ to $\hH_\aA \tensor V_\sigma$ and $\hH_\aA \tensor \hH_\sigma$ in domain and codomain, respectively, gives an operator that is equal to $s(\sigma)$. It follows that $s(\sigma) \in \aA \tensor \End( V_\sigma,\hH_\sigma)$, and hence that $\alpha_g\bigl(s(\sigma)\bigr)=s(\sigma) (\one_\aA \tensor \sigma_g)$ for all $g \in G$ due to Equation~\eqref{eq:SOPequivariance}. 
	As $\sigma$ was arbitrary and the \Star-representation $\pi$ assumed to be faithful, \cite[Lem.~3.2]{SchWa17} now implies that $(\aA, G, \alpha)$ is free, and this is precisely the desired conclusion. 
	\qedhere
\end{proof}

\section*{Acknowledgement}

This research was supported through the program ``Research in Pairs'', \href{https://owpdb.mfo.de/detail?photo_id=23078}{RiP (1910p)}, by the Mathematisches Forschungsinstitut Oberwolfach in 2019. 
The authors also wish to thank the Centre International de Rencontres Math\'ematiques, REB (2177), and Blekinge Tekniska H\"ogskola for their financial support in facilitating this collaboration.
The first name author is indepted to iteratec GmbH.
Finally, the authors wish to thank  Walter van Suijlekom for pointing out the relation to unbounded $KK$-theory.



\bibliographystyle{abbrv}
\bibliography{short,RS}

\end{document}